\documentclass{amsart} 

\calclayout

\usepackage{amsmath,amssymb,amsthm,enumerate,color}

\usepackage{mathtools}
\usepackage{amssymb}
\usepackage[normalem]{ulem}
\usepackage{xcolor}
\usepackage{amsthm}
\usepackage{amsmath}
\usepackage{hyperref}
\usepackage{multicol}
\usepackage{array}
\usepackage{enumerate}
\usepackage{graphicx}
\usepackage{calc}
\usepackage{array}
\newcolumntype{P}[1]{>{\centering\arraybackslash}p{#1}}
\newcolumntype{M}[1]{>{\centering\arraybackslash}m{#1}}

\newcommand{\cd}{\partial}
\newcommand{\cdbar}{\bar{\partial}}

\newcommand{\cdbarst}{\bar{\partial}^*}

\newcommand{\bC}{\mathbb C}

\newcommand{\beq}{\begin{equation}}
\newcommand{\eeq}{\end{equation}}

\newcommand{\norm}{|\!|}
\newcommand{\normm}{|\!|\!|}

\makeatletter
\newcommand*{\DivideLengths}[2]{%
	\strip@pt\dimexpr\number\numexpr\number\dimexpr#1\relax*65536/\number\dimexpr#2\relax\relax sp\relax
}
\makeatother

\title{Local algebraicity and localization of the Bergman kernel on Stein spaces with finite type boundaries}

\date{}

\thanks{The first author was partially supported by the NSF through the grant DMS-2154368. The third author was partially supported by the NSF through the grant DMS-2045104.}

\numberwithin{equation}{section}
\newtheorem{theorem}{Theorem}
\newtheorem*{theorem*}{Theorem}
\newtheorem{proposition}{Proposition}[section]
\newtheorem{remark}{Remark}[section]

\newtheorem{definition}{Definition}[section]

\newtheorem{lemma}{Lemma}[section]

\newtheorem{claim}{Claim}[section]
\newtheorem*{claim*}{Claim}

\begin{document}

\author{Peter Ebenfelt}
\address{Department of Mathematics, University of California San Diego, La Jolla, CA, 92093, USA}
\email{pebenfel@math.ucsd.edu}
\author{Soumya Ganguly}
\address{Department of Mathematics, University of California San Diego, La Jolla, CA, 92093, USA}
\email{s1gangul@ucsd.edu}
\author{Ming Xiao}
\address{Department of Mathematics, University of California San Diego, La Jolla, CA, 92093, USA}
\email{m3xiao@ucsd.edu}

\begin{abstract}
		On a two dimensional Stein space with isolated, normal singularities, smooth finite type boundary, and locally algebraic Bergman kernel, we establish an estimate on the type of the boundary in terms of the local algebraic degree of the Bergman kernel. As an application, we characterize two dimensional ball quotients as the only Stein spaces with smooth finite type boundary and locally rational Bergman kernel.
A key ingredient in the proof of the degree estimate is a new localization result for the Bergman kernel of a pseudoconvex, finite type domain in a complex manifold.
	\end{abstract}
	
	\maketitle


\section{Introduction}

The purpose of this paper is two-fold. The first objective, and the original motivation for the work, is to generalize results from \cite{ebenfelt2021algebraicdegfinitetype} regarding algebraic Bergman kernels on domains in $\bC^2$ to domains in 2-dimensional Stein spaces. The second, needed to resolve the first, is to establish a result on localization of the Bergman kernel on domains with smooth finite type boundaries in complex manifolds. The second objective is of independent interest and generalizes the classical work of 
Fefferman \cite{Feffermanasympbyfeff1974} for strictly pseudoconvex domains in $\bC^n$, as well as other work in more general situations (see below). The proof of the localization result follows the general approach in \cite{Feffermanasympbyfeff1974, HuangLiBergmanEinstein} but requires some new technology, including the introduction of a new negative Sobolev norm.

In complex analysis and geometry, the Bergman kernel and metric have been fundamental objects for our understanding of the geometry  of complex manifolds since the classical work of Bergman and  \cite{BergmanpaperonBergmankernel1}, \cite{BergmanpaperonBergmankernel2} and Kobayashi \cite{GeometryboundeddomKobayashi1959}.
For example, the resolution of Cheng's Conjecture \cite{FuWongChengConjdim2}, \cite{Nemirovski_2006_ConjChengRama}, \cite{HuangXiao2021Chengdim3} characterizes the unit ball $\mathbb{B}^n \subset \mathbb{C}^n$, up to biholomorphisms,  as the only smoothly bounded strictly pseducoconvex domain in $\bC^n$ with a Bergman metric that is also K\"ahler-Einstein.  Further developments on generalizations of Cheng's Conjecture can be found in, e.g., \cite{HuangLiBergmanEinstein}, \cite{ebenfel2020classificationstein}, \cite{gangulysinhachengstein2d} and \cite{savalexiao2023kahlereinstein}.
In a different direction,  together with H. Xu the first and third authors \cite{ebenfelt2020algebraicity, ebenfelt2021algebraicdegfinitetype} studied bounded pseudoconvex domains $G$ in $\bC^2$ whose Bergman kernels are algebraic. In \cite{ebenfelt2020algebraicity}, it is proved that if the boundary of $G$ is strongly pseudoconvex and its Bergman kernel is algebraic, then there is an algebraic biholomorphism from $G$ to $\mathbb{B}^2$. This conclusion fails if the strong pseudoconvexity of $\partial G$ is dropped in view of calculations due to D'Angelo \cite{DA78} showing that complex ellipsoids of the form
\begin{equation*}
|z|^2+|w|^{2p}<1
\end{equation*}
have algebraic Bergman kernels of degree $p$. Nevertheless, in \cite{ebenfelt2021algebraicdegfinitetype}, it is proved that a smoothly bounded pseudoconvex domain $G$ (without any additional assumptions on the boundary $\partial G$) in $\mathbb{C}^2$ has a {\em rational} Bergman kernel if and only if there is a rational biholomorphism from $G$ to $\mathbb{B}^2$. To prove this, a key step is to establish an optimal bound for the type of the boundary of the domain in terms of the algebraic degree of the Bergman kernel, assuming both of these quantities are finite. We mention that, in general, there are several different notions of type of a boundary point such as D'Angelo $k$-type and commutator type (or H\"ormander--Kohn--Bloom--Graham type). In two dimensions, however, all these notions coincide (see, e.g., \cite{ebenfeltCRgeometrybook}).

In this paper, we generalize the aforementioned results in \cite{ebenfelt2021algebraicdegfinitetype} to domains in 2-dimensional Stein spaces with isolated singularities. The following is the first main result of this paper.

\begin{theorem}\label{typedegreeestimate}
  Let $\Omega'$ be  a two dimensional, normal Stein space  with isolated singularities and let $\Omega$ be a precompact, pseudoconvex domain in $\Omega'$ with smooth and finite type boundary $\partial \Omega$.

  \begin{enumerate}
  	
   \item [{\rm (i)}]If $\xi$ is a boundary point of $\Omega$ and there is a local coordinate chart of $\Omega'$ at $\xi$ in which the Bergman kernel form $K_\Omega(z,\bar z)$ of $\Omega$ is algebraic with algebraic degree $d$, then the type $r(\xi)$ of the boundary point $\xi$ satisfies $r(\xi) \leq 2 d$.

    \item [{\rm (ii)}] If, for every boundary point $\xi$ of $\Omega$, there is a local coordinate chart of $\Omega'$ at $\xi$ in which
    the  Bergman kernel form $K_\Omega(z,\bar z)$ of $\Omega$ is rational (algebraic of degree $d=1$), then $\Omega$ is biholomorphic to a finite ball quotient $\mathbb{B}^2/\Gamma$, where $\Gamma$ is a finite, fixed point free, unitary subgroup of $Aut(\mathbb{B}^2)$.

\end{enumerate}

\end{theorem}

\bigskip

We make several remarks. First, we shall refer to a Bergman kernel form $K=K_\Omega(z,\bar z)$ (on the diagonal) such that there is a local coordinate chart at $\xi\in\partial\Omega$ in which $K$ is algebraic with algebraic degree $ \leq d$ as {\em locally algebraic} at $\xi$ of degree $\leq d$. It is (essentially) proved in \cite{ebenfelt2020algebraicity} that if the boundary $\partial\Omega$ is strongly pseudoconvex and $K$ is locally algebraic, then the log-term in Fefferman's asymptotic expansion of $K$ vanishes to infinite order at $\partial\Omega$; thus, having a locally algebraic Bergman kernel form is a strong  (non-generic) condition. Second, we mention that the degree estimate in Theorem \ref{typedegreeestimate} (i) is sharp in view of the result by D'Angelo ({\em loc.\ cit.}) on the Bergman kernels of the complex ellipsoids mentioned above. Third, the converse in Theorem \ref{typedegreeestimate} (ii), namely that ball quotients have locally rational Bergman kernels was shown in
\cite{ebenfel2020classificationstein} (see also \cite{HuangLiBergmanEinstein}). Fourth, by an example in \cite[Section 6]{ebenfelt2020algebraicity}, the conclusion in Theorem \ref{typedegreeestimate} (ii) fails in higher dimensions; The example in \cite[Section 6]{ebenfelt2020algebraicity} is a domain in a 3-dimensional Stein space (with a single isolated singularity) that is not biholomorphic to a ball quotient, but whose Bergman kernel can be shown to be locally rational at every boundary point (a minor detail is left to the reader). We finally note that, as pointed out in \cite{ebenfelt2021algebraicdegfinitetype}, one cannot expect to have an upper bound on the algebraic degree of Bergman kernel in terms of finite type of the boundary. This is because the latter is a biholomorphic invariant while the algebraic degree is not, and can be made arbitrarily large by biholomorphic transformations. 

In the proof of Theorem \ref{typedegreeestimate}, besides utilizing techniques from \cite{ebenfelt2020algebraicity, ebenfelt2021algebraicdegfinitetype},  a key ingredient is to establish a localization result for the Bergman kernel of complex manifolds with pseudoconvex and finite type boundary. This result is of independent interest and is established for complex manifolds of any dimension $n\geq 2$. The result can be formulated as follows.

\begin{theorem}\label{thm:BergmanLoc}
Let $M'$ be a complex manifold of dimension $n\geq 2$ and $M\subset M'$ a precompact, pseudoconvex domain with smooth and finite D'Angelo 1-type boundary $\partial M$. Let $\xi\in \cd M$ and $U\subset M'$ be an open neighborhood of $\xi$. Then, there exists a domain $D\subset M$ with smooth and finite D'Angelo 1-type boundary, an open neighborhood $V\subset U$ of $\xi$ such that $D\cap V=M\cap V$, and a $C^\infty$-smooth $(n,n)$-form $\phi(z,\bar z)$ on $V$ such that
\begin{equation}
K_M(z,\bar z)=K_D(z,\bar z)+\phi(z,\bar z),\qquad z\in D\cap V.
\end{equation}
\end{theorem}

A localization result for the Bergman kernel, as in Theorem \ref{thm:BergmanLoc}, was first established by Fefferman \cite{Feffermanasympbyfeff1974} for smoothly bounded, strictly pseudoconvex domains in $\bC^n$. This result was extended by Engli\v s \cite{Englispseudolocalestimateondbargeneralpscvx} to hold on pseudoconvex domains in $\bC^n$ near finite type boundary points $\xi$. More recently, X. Huang and X. Li extended Fefferman's result to precompact domains with smooth strictly pseudovonvex boundaries in Stein manifolds. For our application in the proof of Theorem \ref{typedegreeestimate}, we need the generality afforded by Theorem \ref{thm:BergmanLoc}.

Theorem \ref{thm:BergmanLoc} follows immediately from Lemma \ref{internaldomtouch} and Theorem \ref{locBergker} below. For the proof, we first follow the framework of \cite{Feffermanasympbyfeff1974, HuangLiBergmanEinstein}. We need to introduce some new ideas for this approach to work in the finite type case. In particular, we introduce a new Sobolev-type norm for negative indices, which is used in our proof.

The paper is organized as follows.   We first prove the localization result Theorem \ref{thm:BergmanLoc} for the Bergman kernel in Section 2. Some crucial steps in this proof are established in three appendices (A, B, and C at the end of the paper).
We then give a proof of Theorem \ref{typedegreeestimate} in Section 3.

	\subsection*{Acknowledgement}
     The authors are grateful to Ioan Bejenaru, Harold Boas, Xiaoshan Li, Mei-Chi Shaw, Ziming Shi, Liding Yao, and Emil Straube for important insights regarding $\bar{\partial}$-Neumann problem, techniques in PDEs, and harmonic analysis.

	\section{Localization of the Bergman kernel}\label{genhuangli}
	In this section we prove a localization result for the Bergman kernel of a complex manifold with pseudoconvex and smooth, finite type boundary. There is such a localization result in \cite{Englispseudolocalestimateondbargeneralpscvx} for finite type, pseudoconvex domains in $\mathbb{C}^n$. However the approach in that work cannot be directly carried over to domains in complex manifolds and for our application purpose, we need to extend the result to the manifold case. Huang-Li already established a localization result for the Bergman kernel of complex manifolds with strongly pseudoconvex boundary in  \cite{HuangLiBergmanEinstein}. For our proof, we shall follow the basic approach in \cite{HuangLiBergmanEinstein}, but make necessary changes and adjustments for the context of finite type. We recall that there are two standard notions of finite type, one due to H\"ormander, Kohn, Bloom--Graham and the other due to D'Angelo. 
These notions coincide for a real hypersurface in $\mathbb{C}^2$ or, more generally, in a complex manifold of dimension two. Our localization result here is for $n \geq 2$ dimensions and we will consider finite D'Angelo 1-type, as it gives us the necessary and sufficient condition for existence of subelliptic estimates of the $\cdbar$-Neumann problem \cite{Catlin1983necessarysubellipticest, CatlinSubellipticestimate}
(see also \cite{CatlinDangelosubellpticestimate}). The precise definition of D'Angelo 1-type (henceforth, simply referred to as ``finite type") is not needed for our purposes, only the subelliptic estimates, and the interested reader is referred to \cite{D'AngeloBook} (see also \cite{ebenfeltCRgeometrybook}) for the definition.

We first recall some necessary preliminaries about complex manifolds and Bergman kernels. Let $M$ be a precompact domain in an $n$ dimensional complex manifold $M'$ with smooth pseudoconvex boundary of finite type (in the sense of D'Angelo). We fix a Hermitian metric $g'$ on the complex manifold $M'$. This gives us a pointwise inner product on $(p,q)$ forms and, more generally, on any tensor product of the tangent and cotangent bundles. We let $\nabla$ be the Levi-Civita connection of $g$. Fixing a metric and a connection enables us
to introduce natural Sobolev spaces for $(p,q)$ forms on $M$. We first define the $L^2$ inner product on the tensor bundles by $(f,h):=\int_M \langle f,h\rangle_{g'} dV_{g'}$, where $\langle\cdot,\cdot\rangle_{g'}$ denotes the pointwise inner product of the tensors $f$ and $h$ given by the metric $g'$. The corresponding $L^2$-norm will be denoted by $\norm f\norm_{0}=(f,f)^{1/2}$. We then define, for integers $s\geq 0$, the Sobolev norm $\norm f\norm^2_{s}=\sum_{k=0}^{s}\norm \nabla^k f\norm^2_{0}$.
We note that the norms so defined depend on the metric $g'$, but the norms corresponding to different metrics are all equivalent. For a $(p,q)$ form $f$ whose support is contained in a coordinate chart $(U,\psi)$, we note that the Sobolev $s$-norm is equivalent to that given by the standard Sobolev $s$-norm in $\bC^n$:
\beq
\norm f\norm^2_{s}\sim \sum_{|I|+|J|\leq s}\sum_i (-1)^{n^2/2}\int_{\psi(U\cap M)} \left|\frac{\partial^{|I|+|J|}}{\partial z^I\partial\overline {z^J}}(f_i\circ\psi^{-1})\right|^2 dz\wedge d\bar z,
\eeq
where $z$ is the local coordinate, $dz\wedge d\bar z=dz_1\wedge\ldots\wedge dz_n\wedge d\bar z_1\wedge\ldots\wedge d\bar z_n$, $f=\sum_if_i\omega^i$, and $\{\omega^i\}$ is a pointwise orthonormal basis of $(p,q)$ forms with respect to the metric $g'$ over $U$. For non-integral $s\geq 0$, we can define Sobolev norms of order $s$ by using a partition of unity and the standard Sobolev $s$-norms on $\bC^n$.
Let $\{U_{\gamma}\}_{\gamma \in \Gamma}$ be a locally finite covering of $M'$ by charts with coordinate mappings $\psi_{\gamma}: U_{\gamma} \to \mathbb{C}^n$, and let $\{\xi_{\gamma}\}$ be a partition of unity subordinate to $\{U_{\gamma}\}$. Let $\{\omega_{\gamma}^i\}$ be a pointwise orthonormal basis of $(p,q)$ forms with respect to the metric $g'$ over $U_{\gamma}$. Then, for $f$ a smooth $(p,q)$ form in $\overline{M}$, we define $\norm f\norm^2_{s}=\sum_{\gamma}\sum_{i \in I}\norm(\xi_{\gamma}f_i)\circ \psi_{\gamma}^{-1}\norm_s^2$, where $f=\sum_{i \in I}f_i \omega_{\gamma}^i$ in $U_{\gamma}$ and the latter $\norm\cdot\norm_s$ denotes the standard Sobolev $s$-norm in $\bC^n$.  The norms introduced in the latter way are less intrinsic and, for integral $s\geq 0$, do not coincide with, but are equivalent to, the more intrinsic ones introduced using the $L^2$ inner product and Levi-Civita connection above. Interested readers can find more details in the Appendix of \cite{FollandKohndbarneumanntheorybook}. We now define, for $s \geq 0$, the Sobolev spaces $H^{(p,q)}_s(M)$ to be the completion of the smooth $(p,q)$ forms in $\overline{M}$ with respect to $\norm{}\cdot \norm{}_{s}$. For $s=0$, we will often denote $H^{(p,q)}_0(M)$ by $L^2_{p,q}(M)$, which we will regard as a Hilbert space endowed with the inner product $( \cdot , \cdot )$ defined above. We may omit the $(p,q)$ in the notation and just denote the Sobolev spaces by $H_s(M)$ when the type of the form is clear.

Let $\Omega^{p,q}(\overline{M})$ be the space of smooth $(p,q)$ forms on $M$, which are smooth up to the boundary, and $\Omega^{p,q}_c(M)$ the subspace of $\Omega^{p,q}(\overline{M})$ consisting of forms having compact support in $M$.
Recall that, for a $(p,q)$ form $g$, the standard negative Sobolev norm, for $s>0$, is given by
\begin{align*}
	\norm{}g\norm{}_{-s}=\sup_{h \in \Omega_c^{p,q}(M)} \frac{|(g,h)|}{\norm{}h\norm{}_{s}}.
\end{align*}	
We also introduce a new negative Sobolev-type norm as
\begin{align*}
\norm{}g\norm{}_{-s}^{\dagger}=\sup_{h \in \Omega^{p,q}(\overline{M}) \cap Dom(\cdbarst)} \frac{|( g,h)|}{\norm{}h\norm{}_{s}}.
\end{align*}	
If we compare these two notions of negative Sobolev norms, we easily see that $\norm{}g\norm{}_{-s} \leq \norm{}g\norm{}^{\dagger}_{-s}$. We define $H^{\dagger}_{-s}(M)$ to be the completion of $\Omega^{p,q}(\overline{M})$ under $\norm{} . \norm{}^{\dagger}_{-s}$ norm. In particular, $H^{\dagger}_{-s}(M)$ includes every $f \in L^2_{p,q}(M)$. 
	
For $(n,0)$ forms, we can also get a metric independent description of the Sobolev space $L^2_{n,0}(M)$.  Note that $\Omega^{n,0}_c(M)$ comes with a natural $L^2$ inner product structure on it as follows:
	\begin{align}\label{l2innerproduct}
		(f,g)'=(-1)^{n^2/2}\int_M f \wedge \bar{g} \ \ \ \text{for all} \ f,g \in \Omega^{n,0}_c(M).
	\end{align}
It is not difficult, and left to the diligent reader, to check that the two inner products $(\cdot,\cdot)$ and $(\cdot,\cdot)'$ coincide for $(n,0)$ forms. We shall henceforth use only the notation $(\cdot,\cdot)$.
To define the Bergman kernel form, we proceed as follows. If we write $\Lambda^n(M)$ to be the space of holomorphic $(n,0)$ forms on $M$  equipped with the norm coming from \eqref{l2innerproduct}, then we can define the Bergman space on $M$ to be:
	\begin{align*}
		A^2(M) \coloneqq \{f \in \Lambda^n(M): \norm{}f\norm{}_{0} < \infty \}.
	\end{align*}
	One can show that $A^2(M)$ is a closed subspace of $L^2_{n,0}(M)$. The orthogonal projection $P: L^2_{n,0}(M) \to A^2(M)$ is called the Bergman projection on $M$. The reproducing kernel of this projection is known as the Bergman kernel (form) on $M$ which we denote by $K_M(z, \bar{w})$ for $z,w \in M$. Note that $K_M(z, \bar{z})$ is an $(n, n)$ form on $M$, and can also be equivalently expressed (as a series) in terms of an orthonormal basis of  $A^2(M)$. Fix  two local, holomorphic coordinate charts $(U,z)$ and $(V, w)$ of $M'$ around $z^* \in \bar{M}$ and $w^* \in \bar{M}$. We write  $z=(z_1,z_2,....z_n)$, $w=(w_1,w_2,....w_n)$, and $dz =dz_1 \wedge dz_2... \wedge dz_n$ and likewise for $d\overline{z}, dw$ and $d \overline{w}.$ Then within the charts $(U, z) \times (V, w)$, the Bergman kernel can be written
	 $$K_M(z,\overline{w})=k_M(z,\overline{w})dz \wedge \overline{dw}.$$
We write $K_M(z,\overline{w^*})=k_M(z,\overline{w^*})dz \wedge \overline{dw} |_{w^*}$ and $\tilde{k}_M(z,\overline{w^*})=k_M(z,\overline{w^*})dz.$
Note for a fixed $w_0 \in M$, $\tilde{k}_M(z,\overline{w_0})$ is a $(n,0)$ form on $M$, and we shall say $K_M(z,\overline{w_0})$ is $L^2$ integrable with respect to $z$ if:
	\begin{align*}
		(-1)^{n^2/2}\int_M \tilde{k}_M(z,\overline{w_0}) \wedge \overline{\tilde{k}_M(z,\overline{w_0})} < \infty .
	\end{align*}
Note the above notion of $L^2$ integrability with respect to $z$ does not depend on choice of local coordinate $w$.
	
We next state a lemma which will be needed for the setup of the localization result for the Bergman kernel. The lemma can be proved identically as Proposition 1.2 in \cite{Englispseudolocalestimateondbargeneralpscvx} (see also the lemma on page 470 in \cite{BelldiffBergker}), therefore we omit the proof.
	\begin{lemma} \label{internaldomtouch}
Let $p \in \partial M$ (which is in particular of finite type) and $(U,z)$ be a small holomorphic coordinate chart at  $p$ in $M'$. Then there exists a smoothly bounded pseudoconvex domain $D \subset U \cap M$ of finite type satisfying:
		\begin{align*}
			D \cap B(p,2 \delta)=M \cap B(p, 2 \delta),
		\end{align*}
		for some sufficiently small $0<\delta <1$, where $B(p, \epsilon)=\left\{q \in U: |z(q)|=\sqrt{\sum_{j=1}^{n}|z_j(q)|^2} < \epsilon \right\}$.
	\end{lemma}

\begin{theorem}\label{locBergker}
Let  $p \in \partial M$ and $D$ be as in Lemma \ref{internaldomtouch} and let $\delta$ be sufficiently small. Then the following  holds in
$B(p,\delta)$ for some $\phi(z, \bar{z}) \in C^{\infty}(B(p,\delta) \cap \overline{M}):$
		\begin{align}
			k_M(z,\bar{z})=k_D(z,\bar{z})+ \phi(z, \bar{z}) \ \text{for}\ z \in B(p,\delta) \cap M.
		\end{align}

	\end{theorem}
	
	\begin{proof} The proof uses the general theory and framework of the $\bar\partial$-Neumann problem and follows the approach in the proof of Proposition 3.1 in \cite{HuangLiBergmanEinstein}. A reader not familiar with the $\bar\partial$-Neumann problem is referred to, e.g., \cite{FollandKohndbarneumanntheorybook} for an introduction. We split the proof into two steps.
\medskip

\noindent\textbf{Step 1:} Fix the holomorphic chart $(U, w)$ at $p$ as specified in Lemma \ref{internaldomtouch}, and as before, we write $d\bar{w}|_w=d\bar{w}_1\wedge \cdots \wedge d\bar{w}_n|_w$ for  $w \in U$. Fix some $w \in B(p, \delta) \cap M$ and set
		\begin{align*}
			f_w(z) \coloneqq K_M(z,\bar{w})-K_D(z,\bar{w}) \chi_D(z) \ \text{for all} \ z \in M.
		\end{align*}
	    where $\chi_D$ is the characteristics function of $D$. We write $f_w(z) \coloneqq \tilde{f}_w(z) \wedge d\bar{w}|_w$ and  $\tilde{g}_w(z) \coloneqq \bar{\partial}\tilde{f}_w(z)$. Thus, $\tilde{f}_w(z)$ is an $L^2$ integrable (with respect to $z$ coordinate) $(n,0)$ form on $M$. By the reproducing property of Bergman kernel, $\tilde{f}_w(\cdot) \perp A^2(M)$ and $\tilde{g}_w$ is a $(n,1)$ form in the sense of distributions. We note that $\tilde{g}_w \in H^{\dagger}_{-1}(M)$, $supp(\tilde{g}_w(z)) \subset \cd D \setminus \cd M$, and  $(\tilde{g}_w, h)=(\cdbar \tilde{f}_w, h) \coloneqq (\tilde{f}_w(z), \cdbarst h)$ for $h \in \Omega^{n,0}(\overline{M}) \cap Dom(\cdbarst)$. 
	    
	    Let us write $\tilde{f}_w(z) \coloneqq \breve{f}_w(z) dz|_z$. If $M$ is defined by a smooth defining function $\{ \rho < 0 \}$, let $\hat{N}$ be the unit outward normal at $p$, i.e., $\vec{N}/||\vec{N}||$ where $\vec{N}=(\frac{\cd \rho}{\cd \bar{z}_1}(p), \ldots, \frac{\cd \rho}{\cd \bar{z}_n}(p))$. Let us choose $\psi(z)$ as the signed-Euclidean distance function to $\Gamma=\overline{\cd D \setminus \cd M}$ in  a small neighborhood $V$ of $\Gamma$ and we extend this smoothly to $M'$ such that $\psi<0$ in $D$, $\psi > 0$ in $M \setminus \bar{D}$, and $|\psi| \geq \epsilon_0$ outside this neighborhood $V$, for some $\epsilon_0>0$ small enough. We choose a sequence of positive numbers $\{\epsilon_j\}_{j \in \mathbb{N}}$ such that $\epsilon_j = 2^{-(j+1)} \epsilon_0$. Then we define, for $j \in \mathbb{N}$ large enough, $D_{\epsilon_j} \coloneqq \{\psi < - \epsilon_j\}$, and 
	    \begin{align*}
	    	h^j_w(z) \coloneqq \begin{cases}
	    		\breve{f}_w(z - \frac{\epsilon_0}{2^{j+3}}\hat{N}), \ z \in \overline{D_{\epsilon_j} \cap M}, \\
	    		\breve{f}_w(z), \ z \in  M \setminus \bar{D}, \\
	    		0 \ \text{otherwise.}
	    	\end{cases}
	    \end{align*} 
	    	Let $\{\chi_j\}$ be a sequence of smooth cutoff functions where for each $j$, $1-\chi_j$ is supported in $U$. Moreover we set $\chi_j$ to be identically 1 in $\{\psi \geq \epsilon_j\} \cup \{\psi \leq -2\epsilon_j\}$ and $0$ in $\{-\epsilon_j \leq \psi \leq \epsilon_j/2\}$. From here we define a sequence of $(n,0)$ forms as $\tilde{f}_{w,j}(z) \coloneqq h^j_w(z)\chi_j(z) dz$ that are smooth up to $\bar{M}$. It is easy to see that $\tilde{f}_{w,j} \to \tilde{f}_w$ in $L^2_{n,0}(M)$. Set $\tilde{g}_{w,j}(z)\coloneqq \bar{\partial}\tilde{f}_{w,j}.$ We note that $supp(\tilde{g}_w(z)) \subset \cd D \setminus \cd M$, and by our construction, $supp(\tilde{g}_{w,j}(z))$ is contained in $\{-2 \epsilon_j < \psi(z) < \epsilon_j\} \cap M$ for $j$ large enough. This implies that the Euclidean distance between boundary of $supp(\tilde{g}_{w,j}(z))$ from $supp(\tilde{g}_w(z))$ is at most $2^{-j}\epsilon_0$, for $j$ large enough.
	    
	    Moreover, we have
	    \begin{align}\label{convergegepsilon}
	    	\tilde{f}_{w, j} \to \tilde{f}_w \ \text{in}\ L^2_{n,0}(M) \implies \tilde{g}_{w,j} \to \tilde{g}_w \ \text{in}\ H^{\dagger}_{-1}(M).
	    \end{align}
        To see this, we take any $h \in \Omega^{n,0}(\overline{M}) \cap Dom(\cdbarst)$ and note $$|(\tilde{g}_{w,j}-\tilde{g}_w, h)| = |(\cdbar(\tilde{f}_{w,j}-\tilde{f}_w), h)| = |(\tilde{f}_{w,j}-\tilde{f}_w, \cdbarst h)| \leq \norm{}\tilde{f}_{w,j}-\tilde{f}_w\norm{}_0 \, \norm{}\cdbarst h\norm{}_0 \leq C' \norm{}\tilde{f}_{w,j}-\tilde{f}_w\norm{}_0 \, \norm{}h\norm{}_{1},$$ where the constant $C'$ is independent of $h$. Hence by the definition of $\norm{}\cdot \norm{}^{\dagger}_{-1}$ norm, we obtain
\begin{equation}\label{eqngwjmgw}
\norm{}\tilde{g}_{w,j}-\tilde{g}_w\norm{}^{\dagger}_{-1} \leq C' \norm{}\tilde{f}_{w,j}-\tilde{f}_w\norm{}_{0} \to 0,~\mathrm{as}~j \to \infty.
\end{equation}
    Now we need to prove the following pseudolocal estimate for the $\cdbar$-Neumann operator:
	
\begin{claim}\label{claim21}
	    		Let $N: L^2_{n,1}(M) \to L^2_{n,1}(M)$ be the $\bar{\partial}$-Neumann operator. Let the point $p \in \partial M$ have finite type $m$ and $g =\cdbar f$, for a smooth $(n,0)$ form $f$, be compactly supported in $U \cap \overline{M}$. Then the following pseudolocal estimate holds:
	    	\begin{align} \label{pseudolocalestimateneg1norm}
	    		\norm{}\xi N g \norm{}_{s+\lambda} \leq C_s (\norm{}\xi_1 g \norm{}_{s} + \norm{} g \norm{}^{\dagger}_{-1}), \forall s \in \mathbb{N} \cup \{0\},
	    	\end{align}
for any $\xi, \xi_1 \in C^{\infty}_c(B(p, \frac{3}{2}\delta))$ that satisfy $\xi_1|_{W} \equiv 1$ and $\xi|_{B(p, \delta)} \equiv 1$, where $W  \subset B(p, \frac{3}{2}\delta)$ is a small neighborhood of $supp(\xi)$. The constants $\{C_s\}$ depend on $s$ and $\xi$, but not on $W$ and $\xi_1$. In the two dimensional case, $\lambda$ can be taken as $2^{-m}$ (by \cite{Kohnboundarybehavdbarweakpscvxmandim2}) and $\leq 1/m$ in higher dimensions (by \cite[Theorem 3.3]{CatlinSubellipticestimate}).
\end{claim}

\begin{remark}
\normalfont
These pseudolocal estimates essentially follow from the subelliptic estimates of the $\bar{\partial}$-Neumann operator, as mentioned in \cite{FollandKohndbarneumanntheorybook}.  In two dimensions, subelliptic estimates are proven in \cite{Kohnboundarybehavdbarweakpscvxmandim2}. In higher dimensions, subelliptic estimates for D'Angelo finite type domains in $\mathbb{C}^n$ were proved by Catlin (see \cite{CatlinDangelosubellpticestimate}, \cite{CatlinSubellipticestimate}); these are also valid in our manifold setting because of their local nature. In the following we give a proof of how these pseudolocal estimates follow from the subelliptic ones for manifolds. We also note that the proof of existence of the operator $N$ for $(n,1)$ forms is similar to that for $(0,1)$ forms.
    \end{remark}
       \begin{proof}[Proof of Claim \ref{claim21}:]
        To start, we will need the following well-known estimate:
       \begin{align}\label{localestimate}
       	  \norm{}\xi N g\norm{} _{s+\lambda} \leq C_s \norm{}\xi\norm{}_{c} (\norm{}\xi_1 g\norm{}_{s} + \norm{}\xi_1 N g\norm{}_{s}), \forall s \in \mathbb{N} \cup \{0\},
       \end{align}
       where the constant $C_s$ depends on $s$ in general and $c$ is a high enough Sobolev index depending on $s$ as well. This estimate is mentioned in \cite[Equation~(2)]{BoasExtensionofKerzman} and it essentially follows from the subelliptic estimates of the $\cdbar$-Neumann problem. We include a detailed proof of the estimate (\ref{localestimate}) in Appendix \ref{appendixA}.

       With the estimate in (\ref{localestimate}), we then invoke an iteration scheme by Boas (Appendix A in \cite{BoasiterationpaperSzegoregdom}) to establish \eqref{pseudolocalestimateneg1norm}. For the reader's convenience, we explain the details of this scheme in the proof of Lemma \ref{iterationschemeappend} of Appendix \ref{appendixB} (see also part (2) of Remark \ref{remarkB2}).

   To apply Lemma \ref{iterationschemeappend}, we let $\xi, \xi_1$ be as in Claim \ref{claim21}. Let $d:=dist(supp(\xi), supp(1-\xi_1))$ (where the distance is measured with respect to the chart $U$). We note, because of Lemma \ref{internaldomtouch}, $d < \delta <1$. Furthermore, we take $\epsilon=\lambda$, $r=s+\lambda$, $m=1$, $m'=1+2\lambda$, $L=N$, $f=\tilde{g}_{w,j}$, $R(\cdot)=\norm{} \cdot \norm{}_{s}$, $T = id$ and $G(\cdot) = \norm{}\cdot\norm{}^{\dagger}_{-1}$ in Lemma \ref{iterationschemeappend}.
   By \eqref{localestimate}, we see that the inequality \eqref{hypothesisestimate}  holds.  Note that to apply Lemma \ref{iterationschemeappend}, we also need $N$ to be a linear operator from $L^2(M)$ to itself, and that it satisfies \eqref{SobolevboundedL} with $\Omega$ replaced by $M$. This is verified in Proposition \ref{continuitydbarneumodsobolev} in Appendix \ref{appendixC}. Applying Lemma \ref{iterationschemeappend} then gives us \eqref{pseudolocalestimateneg1norm}.
       \end{proof}

    Now by our construction, $B(p, 3 \delta/2) \cap (\cd D \setminus \cd M) = \emptyset$, hence $\xi_1 \tilde{g}_{w,j} \equiv 0$ for $j$ sufficiently large. So \eqref{pseudolocalestimateneg1norm}, applied to $g=\tilde{g}_{w,j}-\tilde{g}_{w,l}$ for sufficiently large $j,l$, implies that
        \begin{align}\label{modifiedpseudolocal}
        	\norm{}\xi N (\tilde{g}_{w,j}-\tilde{g}_{w,l}) \norm{}_{s+\lambda} \leq C_s  \norm{} \tilde{g}_{w,j} -\tilde{g}_{w,l}
        \norm{}^{\dagger}_{-1},\qquad  \forall s \in \mathbb{N} \cup \{0\}.
        \end{align}
    By \eqref{convergegepsilon} and \eqref{modifiedpseudolocal}, we conclude that $\{\xi N \tilde{g}_{w,j}\}$ is a Cauchy sequence is $H_{s+\lambda}(M)$ for any $s \in \mathbb{N} \cup \{0\}$.


We shall use the following well-known result:
    \begin{claim}\label{claim22}
    	For $f \in Dom(\cdbar) \cap L^2_{n,0}(M)$,  $f - P(f) = \cdbar^*N\cdbar f$,
where $P:L^2_{n,0}(M) \to A^2(M)$ is the Bergman projection.
    \end{claim}

    \begin{proof}
Let $\mathcal{H}$ be the space of harmonic forms that are also in $L^2_{n,1}(M)$ and $H:  L^2_{n,1} \to \mathcal{H}$ the orthogonal projection. By the solution of the $\cdbar$-Neumann problem on $M$, we know that there exists a unique, smooth $(n, 1)$ form $\alpha \perp \mathcal{H}$ (with respect to the $L^2$ inner product) such that $\cdbar f= (\cdbar \cdbarst + \cdbarst \cdbar) \alpha$. Since $\mathcal{H}$, $\text{Range}(\cdbar)$ and $\text{Range}(\cdbarst)$ are mutually orthogonal spaces, we can infer that $\cdbar f$, $\cdbar \cdbarst \alpha$ both being in the range of $\cdbar$ implies $\cdbarst \cdbar \alpha = 0$. This means we have $\cdbar f= \cdbar \cdbarst \alpha$ or $\cdbar (f - \cdbarst \alpha) = 0$ or in other words $u \coloneqq  f - \cdbarst \alpha$ is a holomorphic $(n,0)$ form in $L^2_{n,0}(M)$. But we also note that $\cdbar f= (\cdbar \cdbarst + \cdbarst \cdbar) \alpha$ implies $N\cdbar f= N(\cdbar \cdbarst + \cdbarst \cdbar) \alpha = (id -H) \alpha$. Given that we assumed $\alpha \perp \mathcal{H}$, we get $\alpha = N\cdbar f$. So we have $u =  f - \cdbarst N\cdbar f$. Now we claim that $u=P(f)$, i.e., the Bergman projection of $f$. For seeing that, we note $f = u + \cdbarst N\cdbar f$ and for all $v \in A^2(M)$ we have $(v, \cdbarst N\cdbar f) = (\cdbar v, N\cdbar f) = 0$, i.e., $\cdbarst N\cdbar f \perp A^2(M)$.
\end{proof}


By this last claim,  $\tilde{f}_{w,j} - P(\tilde{f}_{w,j}) = \cdbar^*N\cdbar \tilde{f}_{w,j} = \cdbar^*N\tilde{g}_{w,j}$. Then
    \begin{align}
    	\xi(\tilde{f}_{w,j} - P(\tilde{f}_{w,j}))=\xi \cdbar^* N\tilde{g}_{w,j} = \cdbar^*(\xi N\tilde{g}_{w,j})+ [\xi,\cdbar^*](\xi_1 N\tilde{g}_{w,j})
    \end{align}
Using the above and \eqref{pseudolocalestimateneg1norm} applied to $g=\tilde{g}_{w,j}$ for sufficiently large $j$ (cf., \eqref{modifiedpseudolocal}),
we obtain
	\begin{align}\label{modifiedpseudolocal2}
		\norm{}\xi(\tilde{f}_{w,j} - P(\tilde{f}_{w,j}))\norm{}_{s+\lambda} \leq C_{s+1} \norm{} \tilde{g}_{w,j} \norm{}^{\dagger}_{-1}, \forall s \in \mathbb{N} \cup \{0\}.
	\end{align}
    \begin{claim}\label{boundednessgminus1}
    	$\norm{} \tilde{g}_w \norm{}^{\dagger}_{-1}$ is uniform bounded with respect to $w \in B(p,\delta) \cap M$. 
    \end{claim}
	\begin{proof}[Proof of Claim \ref{boundednessgminus1}:]
		Let $\sigma >0$ be small and $\rho \in C^{\infty}(M')$ be a real-valued function so that $\rho$ is compactly supported in a $2 \sigma$ neighborhood of $ \overline{\cd D \setminus \cd M}$ in $M'$, with respect to the Euclidean metric on $U$, which we denote by $V_{2 \sigma}$ and is identically one in $V_{\sigma}$. Write $K_D(z,w)=\tilde{K}_D(z,\bar{w}) \wedge d\bar{w}|_w$ for all $w \in  B(p,\delta) \cap M$.
		Fix any $\phi=\sum_{j=1}^{n}\phi_j dz_1\wedge\cdots \wedge dz_n \wedge d\bar{z}_j \in \Omega^{n,1}(\overline{M}) \cap Dom(\cdbarst)$. Since $supp(\tilde{g}_w) \subset \cd D \setminus \cd M$, by our definition of $\tilde{g}_w$ acting on $\Omega^{n,1}(\overline{M}) \cap Dom(\cdbarst)$,  we have $(\tilde{g}_w , \phi)=( \tilde{g}_w , \rho \phi)$ and
		\begin{align*}
			(\tilde{g}_w , \rho \phi) &= (\cdbar\tilde{f}_w , \rho \phi) \\
			&=(\tilde{K}_D(z,\bar{w})\chi_D(z) , \cdbar^*(\rho \phi)) \\
			&=\int_D k_D(z,\bar{w}) dz_1\wedge\cdots \wedge dz_n \wedge \overline{\cdbar^*(\rho \phi)},
		\end{align*}
	where $\tilde{K}_D(z,\bar{w})=k_D(z,\bar{w}) dz_1\wedge\cdots \wedge dz_n$ in local coordinates. 
	Assuming $\sigma$ is sufficiently small, we observe that $dist(V_{2\sigma},B(p,\delta)) >0$. By Theorem 1 in \cite{BelldiffBergker}, it follows that
	\begin{align}
		\sup_{z \in V_{2 \sigma} \cap D} |k_D(z,\bar{w})| \leq C,\qquad \forall w \in  B(p,\delta) \cap M,
	\end{align}
where $C$ is a constant independent of $w$. Thus, by the last inequality and equation above, we have
	\begin{equation}
		|(\tilde{g}_w , \phi)|=|(\tilde{g}_w , \rho \phi)| \leq C_1 \norm{}\phi\norm{}_{1},\qquad  \forall w \in  B(p,\delta) \cap M,
	\end{equation}
	where the constant $C_1$ does not depend on $w$. By definition of $\norm{}\cdot\norm{}^{\dagger}_{-1}$, the assertion in our claim is proven. 
	\end{proof}

    We observe that $P(\tilde{f}_{w,j}) \to 0$ in $L^2_{n,0}(M)$, because as we noted above, $ \tilde{f}_w \perp A^2(M)$. By Claim \ref{boundednessgminus1} and equation \eqref{modifiedpseudolocal2}, Rellich's Lemma (\cite{FollandKohndbarneumanntheorybook}, Proposition A.2.3, page 124) implies that, for each $w$, after passing to a subsequence, $\xi(\tilde{f}_{w,j_l} - P(\tilde{f}_{w,j_l}))$ converges in $H_s(M)$, for all $s \in \mathbb{N} \cup \{0\}$. By \eqref{convergegepsilon} and the observation that $P(\tilde{f}_{w,j}) \to 0$ in $L^2_{n,0}(M)$,  we conclude that $\xi(\tilde{f}_{w,j_l} - P(\tilde{f}_{w,j_l}))\to\xi \tilde{f}_w$ in $H_s(M)$. Then, we take the limit in \eqref{modifiedpseudolocal2} using this subsequence and conclude that
	\begin{align}
		\norm{}\xi \tilde{f}_w\norm{}_{s} \leq \tilde{C}_s, \ \text{for all}\ s \in \mathbb{N} \cup \{0\},
	\end{align}
	where the constant $\tilde{C}_s$ does not depend on $w \in  B(p,\delta) \cap M$.
	   \\ \textbf{Step 2:} We write $f_w(z) \coloneqq \tilde{f}_w(z) \wedge d\bar{w}|_w$ and $\tilde{g}_w=\cdbar \tilde{f}_w$ in local coordinates as before. 
	   By the reproducing property of Bergman kernels as noted above, for any $h \in A^2(M)$, we have $\int_M f_w(z) \wedge \overline{h(z)} = \overline{h(w)}-\overline{h(w)}=0$,  $\forall w \in B(p,\delta) \cap M$. Differentiating both sides with respect to $\bar{w}$ and taking differentiation under the integral (which follows from arguments before (9) in \cite{KerzmandifferentiabilityBergker}), we find that $D^{\alpha}_{\bar{w}}\tilde{f}_w \perp A^2(M)$, for any $w \in B(p, \delta) \cap M$ and any multi-index $\alpha$. By repeating a similar argument as in Step 1 above, we conclude that
	   \begin{equation}
	   	\norm{}\xi D^{\alpha}_{\bar{w}} \tilde{f}_w\norm{}_{s} \leq \tilde{C}_{s,\alpha},\qquad \forall s \in \mathbb{N} \cup \{0\}.
	   \end{equation}
   As before, the constants $\tilde{C}_{s,\alpha}$ do not depend on $w \in B(p, \delta) \cap M$. By the Sobolev embedding Theorem, it follows that
   \begin{align}
   	|\xi D^{\beta}_z D^{\alpha}_{\bar{w}} (k_M(z,\bar{w}) - k_D(z,\bar{w}))| \leq C_{\alpha, \beta},\qquad \forall \alpha, \beta,\ \forall z, w \in B(p, \delta) \cap M,
   \end{align}
    where the $C_{\alpha, \beta}$ are constants. Since $\xi|_{B(p,\delta)} \equiv 1$, the last inequality implies $k_M(z,\bar{w}) - k_D(z,\bar{w})$ is smooth up to $\big(B(p, \delta) \cap \overline{M} \big) \times \big(B(p, \delta) \cap \overline{M} \big)$ by the Sobolev embedding Theorem again. By taking $z=w \in B(p, \delta) \cap \overline{M},$ the conclusion of Theorem \ref{locBergker} follows.
	\end{proof}

	\section{The type-degree estimate: Proof of Theorem \ref{typedegreeestimate} }\label{pftypedegin}

We now utilize the localization result presented above to prove Theorem \ref{typedegreeestimate}.
 	\begin{proof}[Proof of Theorem \ref{typedegreeestimate}]
 	   	Let $\Omega$ be a precompact domain in a $2$-dimensional, normal Stein space $\Omega'$ with possibly isolated singularities. Also let $\partial \Omega$ be smooth, pseudoconvex and of finite type. Denote the set of singularities in $\Omega'$ by $Sing(\Omega')$. We can blow up the singularities in $\Omega'$, and denote by $\pi: M' \to \Omega'$ the blow down map where $M'$ is a complex manifold. Denoting the exceptional set in $M'$, i.e., $\pi^{-1}(Sing(\Omega'))$ by $E$, we then know that $\pi: M' \setminus E \to \Omega' \setminus Sing(\Omega')$ is a biholomorphism. Denote $\pi^{-1}(\Omega)$ by $M$ where $\partial M$ will be smooth, pseudoconvex and of finite type, just like $\partial \Omega$. Recall that the Bergman kernel $K_{\Omega}$ of $\Omega$ is defined to be that of its regular part. Since the Bergman kernel $K_M$ is not changed after removing a complex subvariety $E$ (see \cite{GeometryboundeddomKobayashi1959}), we have $K_{\Omega}(z,\bar{z})=K_M(z,\bar{z})$, where $z$ is from the regular part of $\Omega$.
 	   	
 	   	Fix a boundary point $\xi \in \cd \Omega$ and suppose that the Bergman kernel of $\Omega$ has algebraic degree $d \in \mathbb{N}$ in some local chart $(U, z)$ around $\xi$. Write the minimal polynomial of $K_{\Omega}$ in local coordinates $z$ as:
 	   \begin{align}\label{minpolbergker}
 	   	P(z,\bar{z},Y)=\sum_{j=0}^{d}\alpha_j(z,\bar{z})Y^j,
 	   \end{align}
 	   where $\alpha_j'$s are real polynomials for $0 \leq j  \leq d$ and $\alpha_{d} \not\equiv 0$.
 	   By the localization result, Theorem \ref{locBergker}, there exists a domain $D \subset  U \cap \Omega$ as in Lemma \ref{internaldomtouch} with $D \cap B(\xi,2 \delta)=\Omega \cap B(\xi, 2 \delta)$ for some small $\delta >0$ and $\phi(z, \bar{z}) \in C^{\infty}(B(\xi,\delta) \cap \bar{\Omega})$, such that in the local coordinates
 	   \begin{align}\label{eqnlocalize}
 	   	k_{\Omega}(z,\bar{z})=k_D(z,\bar{z})+ \phi(z, \bar{z}), ~z \in B(\xi,\delta) \cap \Omega.
 	   \end{align}
 Since the singularities in $\Omega'$ are assumed to be isolated and $\partial \Omega$ is assumed to be smooth, one can choose $U, \delta$ small enough such that there are no singularities in $U \cap \Omega$. Denote by $\rho$  a defining function of $D$, i.e., $D=\{\rho >0\}$, and let $L$ be a small real line segment that intersects $\partial D$ transversally at $\xi$. By Hsiao-Savale \cite[Theorem 2]{HsiaoSavaleBergmanexpansionfinitetype}, if we denote by $r=r(\xi)$ the type of $\cd\Omega$ at the boundary point $\xi$, then the Bergman kernel $k_D(z,\bar{z})$ has the following expansion along $D \cap L:$
 	
 	   \begin{align}\label{HsiaoSavaleexpansion}
 	    k_D(z,\bar{z}) = \rho^{-2-\frac{2}{r}} \big(\sum_{j=0}^{r-1}a_j\rho^{j/r} + C(z, \bar{z})\big) + B(z,\bar{z})\log \rho.
 	   \end{align}
  The $a_j'$s are constants with $a_0>0$, $B(z,\bar{z})$ is a real smooth function on $L$, and $C(z,\bar{z})$ is a function in $D$ near $\xi$ satisfying $C(z,\bar{z})=O(\rho)$ as $z \to \xi$ along $L$. These $a_j'$s, $B$, $C$ may all depend on the line $L$. Combining the expansion of $k_D$ and the localization result above, we see that
 	   \begin{align}\label{HsiaoSavaleexpansionforomega}
 	   	k_{\Omega}(z,\bar{z})
 	   	= \rho^{-2-\frac{2}{r}} \big(\sum_{j=0}^{r-1}a_j\rho^{j/r} + (C(z, \bar{z}) + \phi(z,\bar{z})\rho^{2+\frac{2}{r}})\big) + B(z,\bar{z})\log \rho.
 	   \end{align}
 	   This gives us an asymptotic expansion of $k_{\Omega}$ from that of $k_D$. We still denote the new coefficient $C(z, \bar{z}) + \phi(z,\bar{z})\rho^{2+\frac{2}{r}}$ by $C,$ which again satisfies $C(z,\bar{z})=O(\rho)$ as $z \to \xi$ along $L$. Then by the same arguments as the proof of Theorem 1.5 in \cite{ebenfelt2021algebraicdegfinitetype},  we get $2d  \geq r=r(\xi)$. This finishes the proof of part (i) of Theorem \ref{typedegreeestimate}.


	
To prove part (ii) of Theorem \ref{typedegreeestimate}, we first note that, by the local rationality assumption, the Bergman kernel has algebraic degree one in some local charts. Then, part (i) of Theorem \ref{typedegreeestimate} yields that the boundary of $\Omega$ must be strongly pseudoconvex. By local algebraicity of the Bergman kernel form, using the same arguments as in the proof of Proposition 5.1 of \cite{ebenfelt2020algebraicity}, we obtain that $\partial \Omega$ is locally Nash-algebraic, i.e., it is locally defined by a real algebraic function in some local coordinate chart. Fix $p \in \partial G$ and such a local coordinate chart $(U, z)$ at $p$.  Let $D \subset  U \cap \Omega$  a small smoothly bounded strongly  pseudoconvex domain  with $D \cap B(p,2 \delta)=\Omega \cap B(p, 2 \delta)$ for some small $\delta >0$ (see Lemma \ref{internaldomtouch}). By  Theorem \ref{locBergker},
there exists some $\phi(z, \bar{z}) \in C^{\infty}(B(p,\delta) \cap \bar{\Omega})$ such that
	\begin{align*}
		k_{\Omega}(z,\bar{z}) = k_{D}(z,\bar{z}) + \phi(z,\bar{z}) \ \text{on} \  D.
	\end{align*}
 Let $r$ be a Fefferman defining function of $D$ (see \cite{Feffermanasympbyfeff1974}). By the Fefferman asymptotic expansion \cite{Feffermanasympbyfeff1974}, we have
	\begin{align*}
		k_D(z,\bar{z}) = \frac{\lambda(z,\bar{z})}{r(z)^3} + \psi(z,\bar{z}) \log r(z) \ \text{on} \ D,
	\end{align*}
	where $\lambda, \psi$ are smooth functions on $D$ that extend smoothly across $B(p, \delta) \cap \partial D$. 
Consequently,
	\begin{align}\label{feffasymplocalized}
		k_{\Omega}(z,\bar{z}) = \frac{\lambda(z,\bar{z}) + \phi(z,\bar{z})r(z)^3}{r(z)^3} + \psi(z,\bar{z}) \log r(z) \ \text{on} \  D.
	\end{align}
By local rationality of Bergman kernel of $\Omega$, there exists real valued polynomials $a_1(z,\bar{z}), a_0(z,\bar{z})$ in $\mathbb{C}^2 \cong \mathbb{R}^4$ with with $a_1 \not \equiv 0$ such that
	\begin{align*}
		a_1(z,\bar{z}) k_{\Omega}(z, \bar{z}) + a_0(z,\bar{z}) = 0 \ \text{on} \ D.
	\end{align*}
	Substituting \eqref{feffasymplocalized} into this last equation and using Lemma 5.2 from \cite{ebenfelt2020algebraicity} yields that the coefficient $\psi$ of the logarithmic term vanishes to infinite order at $\partial \Omega$ near $p$. Since $\Omega$ is $2$-dimensional, by the resolution of the Ramadanov Conjecture (\cite{GrahamScalarboundaryinvariantsRamadanov} and \cite{BoutetdeMonvelRamadanov}), it follows that $\Omega$ is locally spherical near $p$. Since the boundary point $p$ was arbitrary, $\partial \Omega$ carries a compact, spherical CR structure. Now, by a theorem of Huang (Corollary 3.3 in \cite{HuangBallquotientCRlinkiisolatedcomplexsing}),  it follows $\partial \Omega$ is CR equivalent to a CR spherical space form $\partial \mathbb{B}^2/\Gamma$ where $\Gamma \subset U(2)$ is a finite group with no fixed points on $\partial \mathbb{B}^2$. By Step 3 of the proof of \cite[Theorem 1.1]{ebenfelt2020algebraicity}, we conclude that $\Omega$ is a  ball quotient $\mathbb{B}^2/\Gamma$. 
	This proves part (ii) and concludes the proof of Theorem \ref{typedegreeestimate}.
\end{proof}
\vspace{-0.2 in}
\begin{remark}
	\normalfont
    In the proof above we applied our localization result, Theorem \ref{locBergker} to obtain a Hsiao-Savale type expansion for the Bergman kernel of the Stein domain $\Omega$, see \eqref{HsiaoSavaleexpansionforomega}. It is possible that the asymptotic expansion of Hsiao-Savale in \cite{HsiaoSavaleBergmanexpansionfinitetype} applies directly in this Stein case with an adjustment of their proof. The authors have not tried to verify this. 
\end{remark}

\appendix	
\section{Proof of inequality \eqref{localestimate}}
\label{appendixA}

As in section \ref{genhuangli}, let $M$ be a smoothly bounded, pseudoconvex domain with finite type (in the sense of D'Angelo) boundary in a Hermitian manifold $(M',g')$ of dimension $n$, where $g'$ is a Hermitian metric on $M'$.  Fix $p \in \partial M$  and  a local coordinate chart $(U,z)$ at $p$ in $M'$ with $z(p)=0$. For ease of calculations, we will always work with a special boundary chart $(U \cap \overline{M})$, \cite[Section 2.2]{Straubebookl2sobolevtheory} around $p$. We shall follow the standard convention of changing the constants in an estimate in subsequent inequalities without renaming the constant in every step. We shall also use the notation $B(p, r) \coloneqq \left\{q \in U: |z(q)|=\sqrt{\sum_{j=1}^{n}|z_j(q)|^2} < r \right\}$, where we only consider $r>0$ so small that $B(p,r)\subset\subset U$.

\begin{claim}\label{claimA1}
	Let $N: L^2_{n,1}(M) \to L^2_{n,1}(M)$ be the $\bar{\partial}$-Neumann operator and $g =\cdbar f$, for a smooth $(n,0)$ form $f$, be compactly supported in $U \cap \overline{M}$. Let $p\in\partial M$ be a point of finite type. Then,  we have the following local estimate for some $\lambda, c >0$:
\begin{align} \label{pseudolocalestimateneg1norm2}
	\norm{}\xi N g\norm{}_{s+\lambda} \leq C_s \norm{}\xi\norm{}_{c} (\norm{}\xi_1 g \norm{}_{s} + \norm{}\xi_1 N g \norm{} _{s}),\qquad  \forall s \in \mathbb{N} \cup \{0\},
\end{align}
for any $\xi, \xi_1 \in C^{\infty}_c(B(p, \frac{3}{2}\delta))$ such that $\xi|_{B(p, \delta)} \equiv 1$ and $\xi_1 \equiv 1$ on a small open neighborhood of $supp(\xi)$. The constant $C_s$  depends only on $s$ and not on $\xi_1$, (and of course not on $\xi$ either, as already indicated by the expression of (\ref{pseudolocalestimateneg1norm2})).
\end{claim}

\begin{remark} {\rm
We may take any $c \geq n+s+1$. In the two dimensional case, $\lambda$ can be taken as $2^{-m}$ (by \cite{Kohnboundarybehavdbarweakpscvxmandim2}) and $\leq 1/m$ in higher dimensions (by \cite[Theorem 3.3]{CatlinSubellipticestimate}), where $m$ is the type of the boundary point $p$.}
\end{remark}

\begin{proof}[Proof of Claim \ref{claimA1}:]
	For this proof, let us fix some $s \in \mathbb{N} \cup \{0\}$. We first mention some key features of a special boundary chart, following \cite[Section~2.2]{Straubebookl2sobolevtheory}. Fix any $p \in \partial M$, and let $\rho$ be a local defining function of $M$ so that $M$ is defined  as $\{\rho<0\}$ near $p$.  Let $L_1, \ldots, L_{n-1}$ be orthonormal vector fields (with respect to the metric $g'$ on $M'$) near $p$ that span $\mathbb{C}T_z^{(1,0)}(\partial M_{\epsilon})$ for $z$ near $p$ and for small enough $\epsilon > 0$. Here $M_{\epsilon}=\{\rho(z) < \epsilon \}$. We can add the complex (1,0) normal field with length 1 with respect to $g'$, $L_n$, which is a smooth multiple of $\sum_{j=1}^n \frac{\partial \rho}{\partial \bar{z}_j}\frac{\partial }{\partial {z}_j}$ near $p$. Let $\omega_1, \ldots \omega_n$ be dual forms to $L_1, \ldots, L_n$. These form an orthonormal basis for $(1,0)$ forms near $p$ and are  called a \textit{special boundary (co)frame}. One can express any $(n,1)$ form on this chart $U$ in terms of a (local) orthonormal basis formed by $\{dW \wedge \bar{\omega}_j\}_{j=1}^n$ where $dW=\omega_{1}\wedge \ldots \wedge\omega_{n}$.
	
	We note that $T=\frac{1}{2i}(L_n -\bar{L}_n)$ is a real transversal (to the CR tangent bundle) tangent vector field  of $\partial M_{\epsilon}$. Moreover, $Q=\frac{1}{2}(L_n+\bar{L}_n)$ gives the real normal vector field and $\{Re(L_1), Im(L_1), \ldots Re(L_{n-1}), Im(L_{n-1}), T\}$ are real tangential vector fields. We will call the set of these tangential vector fields $S_{tang}$. We have the following important lemma, which is well-known in the field (for example, Lemma 2.2 in \cite{Straubebookl2sobolevtheory}). For notational purposes and the reader's convenience, we state and prove it here.
	
	\begin{lemma}\label{straubelemma2.2gen}
		If $M$ has smooth boundary and $h=\sum_{j=1}^n h_{j} dW \wedge \bar{\omega}_j$ is a smooth $(n,1)$ form, compactly supported in $U \cap \overline{M}$, then in a special boundary frame
		\begin{align}\label{normalderintermsofothers}
			Q(h_{j})=\sum_{D_i \in S_{tang}} \sum_{j} A_{i, j}(D_i h_{j}) + \sum_{i<k}B_{i,k} (\cdbar h)_{i,k} + C (\vartheta h)_{n,0} + \sum_{i}F_{i} h_{i}.
		\end{align}
	where $\vartheta$ is the formal $L^2$-adjoint of $\cdbar$ and the coefficients $A_{i, j}, B_{i,k}, C, F_{i} \in C^{\infty}(\overline{M})$. These coefficients do not depend on $h$. Here $(\cdbar h)_{i,k}, i < k,$ denotes the coefficient of $\cdbar h$ in $dW \wedge \bar{\omega}_i \wedge \bar{\omega}_k$ component.
	\end{lemma}
    \begin{proof}
    	First we note that by very similar calculations as in (2.14) and (2.16) in \cite{Straubebookl2sobolevtheory} or page 107 in \cite{ChenandShawPDESCVbook}, we can infer that with respect to a special boundary frame near $p \in \partial M$,
    	\begin{equation}\label{dbardbatsrtgenform}
    		\begin{aligned}
    			&\cdbar h= \sum_{i=1}^n \sum_{j=1}^n (\bar{L}_j h_{i}) \bar{\omega}_j \wedge dW \wedge \bar{\omega}_i + \sum_{i} h_{i} \cdbar(dW \wedge \bar{\omega}_i) \ \ \text{and} \\
    			& \vartheta{h} = (-1)^{n+1} (\sum_{j=1}^n L_j h_{j}) dW + \ \text{terms with no derivatives of }h_{j}\text{s}.
    		\end{aligned}
    	\end{equation}

    Now we break our argument into two cases. For the coefficient $h_{j}$, if we have $n \neq j$ then we write $Q(h_{j}) = \frac{1}{2}(L_n+\bar{L}_n)(h_{j})=\frac{1}{2}(L_n-\bar{L}_n)(h_{j})+ \bar{L}_n(h_{j}) = i T(h_{j}) + \bar{L}_n(h_{j})$. Then using the first identity in \eqref{dbardbatsrtgenform}, we get that $Q(h_{j})$ can be expressed in terms of a $C^{\infty}(\overline{M})$-linear combination of coefficients of $\cdbar h$, $h$ and tangential vector fields (i.e. in terms of $\{\bar{L}_i\}_{i=1}^{n-1}$ and $T$) of $h_{i}$s. On the other hand, if $n = j$ then we write $Q(h_{j}) = \frac{1}{2}(L_n+\bar{L}_n)(h_{j})=\frac{1}{2}(\bar{L}_n-L_n)(h_{j})+ L_n(h_{j}) = -i T(h_{j}) + L_n(h_{j})$. Then using the second identity in \eqref{dbardbatsrtgenform} we get that $Q(h_{j})$ can be expressed in terms of a $C^{\infty}(\overline{M})$-linear combination of coefficients of $\vartheta h$, $h$ and tangential vector fields (i.e. in terms of $\{L_i\}_{i=1}^{n-1}$ and $T$) on $h_{i}$s.
    \end{proof}
	Now, we focus our attention on the subelliptic estimates the $\cdbar$-Neumann problem. For that, we first define tangential Sobolev norms (for more details, readers are referred to Appendix 3 of \cite{FollandKohndbarneumanntheorybook}). Let us write $\mathbb{R}^m_{-}$ for the closed halfspace of points in $\mathbb{R}^m$ where the last coordinate is non-positive. We denote the first $m-1$ components by $t$ and the last component by $r$.
	\begin{definition}
		Suppose that $u \in C^{\infty}_c(\mathbb{R}^m_{-})$. The partial Fourier transform of $u$ is given by
		\begin{align*}
			\tilde{u}(\tau,r)\coloneqq \int_{\mathbb{R}^{m-1}}e^{- t \cdot \tau} u(t,r)dt.
		\end{align*}
	\end{definition}
	\begin{definition}
		Suppose that $u \in C^{\infty}_c(\mathbb{R}^m_{-})$ and $s \in \mathbb{R}$. The tangential differential operator $\Lambda^s$ and the tangential Sobolev norm $\normm{}u\normm{}_s$ of $u$ are defined as
		\begin{align*}
			(\widetilde{\Lambda^s u})(\tau,r)\coloneqq (1+|\tau|^2)^{s/2}\tilde{u}(\tau,r), \ \ \ \ \normm{}u\normm{}_s \coloneqq \norm{}\Lambda^s u\norm{}_{0, \mathbb{R}^m_{-}}.
		\end{align*}
	\end{definition}
	By using Plancherel's theorem, we note that $\normm{}u\normm{}^2_{s} = \int_{-\infty}^0 \int_{\mathbb{R}^{m-1}} (1+|\tau|^2)^{s} |\tilde{u}(\tau,r)|^2 d\tau dr$. From this definition, it is clear that $\normm{}u\normm{}^2_{s} \leq \normm{}u\normm{}^2_{s'}$ whenever $s \leq s'$. In the context of our situation, when we are in special boundary chart near $p \in \partial M$, the coordinates $t$ can be taken as the tangential coordinates and the coordinate $r$ stands for the radial or normal (to $\partial M$ near $p$) coordinate. For $(n,1)$ forms $\phi$ supported on $U \cap \overline{M}$ and expressed as
	\begin{align*}
		\phi = \sum_{j=1}^n \phi_{j} dW \wedge \bar{\omega}_j,
	\end{align*}
we let
	\begin{align*}
		\normm{}\phi\normm{}_s^2 \coloneqq \sum_{j=1}^n\normm{}\phi_{j}\normm{}^2_s.
	\end{align*}
	By the work of Catlin (\hspace{-5 pt} \cite{Catlin1983necessarysubellipticest}, \cite{CatlinSubellipticestimate}), having a finite type boundary (in the sense of D'Angelo) is the necessary and sufficient condition for the following subelliptic estimate. Although in the original works of Catlin, this estimate is established for smoothly bounded, pseudoconvex domains with finite type boundary in $\mathbb{C}^n$, we note that the estimate is purely local in nature and, hence, the same holds for such domains in complex manifolds as well. For a smooth form $u \in Dom(\cdbar) \cap Dom(\cdbarst)$ which is compactly supported in $U \cap \overline{M}$, we have:
	\begin{align}\label{subellipcatlin}
		\normm{}u\normm{}^2_{\lambda} \leq E (\norm{}\cdbar u\norm{}_{0}^2 + \norm{}\cdbarst u\norm{}_{0}^2 +  \norm{}u\norm{}^2_{0}),
	\end{align}
	where $0 \leq \lambda \leq 1/2$ is same as in the assertion of Claim \ref{pseudolocalestimateneg1norm2}. It is well-known \cite{CatlinDangelosubellpticestimate} that for such a form $u$, there is an alternate version of the subelliptic estimate, which is as follows:
	\begin{align}\label{eq:fullsobolevsubellipticseq0}
		\norm{}u\norm{}^2_{\lambda} \leq E (\norm{}\cdbar u\norm{}^2_{0}+ \norm{}\cdbarst u\norm{}^2_{0}+\norm{}u\norm{}^2_{0}).
	\end{align}
	Let $D$ be an arbitrary first order differential operator acting component-wise on $u$. Without loss of generality, we can assume $D$ is either a tangential derivative $D_T$ or the normal derivative $Q$. We first consider the case where $D=D_T$. We note that in this case, $Du$ remains in the domain of $\cdbarst$ (for details, the readers are referred to \cite[page 14]{Straubebookl2sobolevtheory}). Replacing $u$ by $D_Tu$ in \eqref{eq:fullsobolevsubellipticseq0}, we get
	\begin{align*}
		\norm{}D_Tu\norm{}^2_{\lambda}&\leq E (\norm{}\cdbar D_Tu\norm{}^2_{0}+ \norm{}\cdbarst D_Tu\norm{}^2_{0}+\norm{}D_Tu\norm{}^2_{0}).
	\end{align*}
We obtain
	\begin{align*}
		\norm{}\cdbar D_Tu\norm{}^2_{0} &\leq \norm{}D_T \cdbar u\norm{}^2_{0} + \norm{}[\cdbar, D_T]u\norm{}^2_{0},
	\end{align*}
where $[\cdbar, D_T]$ denotes the commutator of $\cdbar$ and $D_T$.
	Since $\cdbar$, $\cdbarst$, and $D_T$ are first order differential operators, the commutators $[\cdbar, D_T]$, $[\cdbarst, D_T]$ are also differential operators of order 1. 
	Consequently,
	\begin{align*}
		\norm{}\cdbar D_Tu\norm{}^2_{0} \leq E (\norm{}\cdbar u\norm{}^2_{1} + \norm{}u\norm{}^2_{1}).
	\end{align*}
A similar estimate follows for $\norm{}\cdbarst D_Tu\norm{}^2_{0}$. These inequalities imply
	\begin{align}\label{Sobolev1subeliptictangential}
	   \norm{}D_T u\norm{}^2_{\lambda} &\leq E (\norm{}\cdbar u\norm{}^2_{1} + \norm{}\cdbarst u\norm{}^2_{1} + \norm{}u\norm{}^2_{1}).
	\end{align}
If we instead let $D=Q$, then by using \eqref{normalderintermsofothers}, \eqref{Sobolev1subeliptictangential}, and  $\lambda \leq 1/2$, we obtain
	\begin{equation}\label{Sobolev1subelipticnormal}
		\begin{aligned}
			\norm{}Q{u}\norm{}^2_{\lambda} &\leq E (\norm{}D_T u\norm{}^2_{\lambda} + \norm{}\cdbar u\norm{}^2_{\lambda} + \norm{}\cdbarst u\norm{}^2_{\lambda} + \norm{}u\norm{}^2_{\lambda}) \\
			&\leq E (\norm{}\cdbar u\norm{}^2_{1} + \norm{}\cdbarst u\norm{}^2_{1} + \norm{}u\norm{}^2_{1}).
		\end{aligned}
	\end{equation}
Thus, by \eqref{Sobolev1subeliptictangential} and \eqref{Sobolev1subelipticnormal}, for a first order differential operator $D$, we have
    \begin{align*}
    	\norm{}D u\norm{}^2_{\lambda} &\leq E (\norm{}\cdbar u\norm{}^2_{1} + \norm{}\cdbarst u\norm{}^2_{1} + \norm{}u\norm{}^2_{1}).
    \end{align*}
    In other words we have $\norm{}u\norm{}^2_{1+\lambda} \leq E (\norm{}\cdbar u\norm{}^2_{1} + \norm{}\cdbarst u\norm{}^2_{1} + \norm{}u\norm{}^2_{1}).$
    Similarly, by induction we obtain for any $s \in \mathbb{N} \cup \{0\}$,
	\begin{align}\label{eq:fullsobolevsubelliptics}
		\norm{}u\norm{}^2_{s+\lambda}&\leq E (\norm{}\cdbar u\norm{}^2_{s}+ \norm{}\cdbarst u\norm{}^2_{s} + \norm{}u\norm{}^2_{s}).
	\end{align}

We now proceed to prove Claim \ref{claimA1}. We observe that if $g$ is a smooth form in $\overline{M}$ then so is $Ng$. This follows from the Sobolev regularity gain property of the operator $N$; see Proposition \ref{continuitydbarneumodsobolev}; see also \cite[Theorem 6.3]{Kohnboundarybehavdbarweakpscvxmandim2}, \cite{CatlinSubellipticestimate}, \cite[Theorem 3.4]{Straubebookl2sobolevtheory}, \cite[Theorem 3.5]{KohnDAngeloSublleipticestimateandfinitetype}. By definition of the $\cdbar$-Neumann operator $N$, we have $Ng \in Dom(\cdbar)\cap Dom(\cdbarst)$, $\cdbar Ng \in Dom(\cdbarst)$ and $\cdbarst Ng \in Dom(\cdbar)$. Since $\xi \in C^{\infty}(\overline{M})$, we further have $\xi Ng \in Dom(\cdbar)\cap Dom(\cdbarst)$ (cf. \cite[page~11-12]{Straubebookl2sobolevtheory} and \cite[page~106]{ChenandShawPDESCVbook}). Thus, we may apply \eqref{eq:fullsobolevsubelliptics} to $u=\xi Ng$, which gives us,
 for all $s \in \mathbb{N} \cup \{0\}$,
	\begin{equation}\label{imp1}
		\begin{aligned}
			\norm{}\xi N g\norm{}^2_{s+\lambda}&\leq E (\norm{}\cdbar (\xi N g)\norm{}^2_{s}+ \norm{}\cdbarst (\xi N g)\norm{}^2_{s} + \norm{}\xi N g\norm{}^2_{s}).
		\end{aligned}
	\end{equation}
For the first term on the right hand side of (\ref{imp1}), by the product rule for $\cdbar$, we have $\cdbar (\xi N g) = \cdbar \xi \wedge Ng + \xi \cdbar Ng$. For the second term, $\cdbarst (\xi N g)$, we rewrite the result of the product rule differently. Suppose that in terms of the special boundary frame on $U \cap \overline{M}$, the $(n,1)$ form $Ng$ can be written locally as $\sum_{j=1}^n h_{j} \bar{\omega}_j \wedge dW$ where $h_{j}$ are smooth functions on $U\cap \overline{M}$. Then, $\xi Ng=\sum_{j=1}^n \xi h_{j} \bar{\omega}_j \wedge dW$ in $U \cap \overline{M}$ and from \eqref{dbardbatsrtgenform}, using standard arguments, we obtain
    \begin{align*}
    	\cdbarst (\xi Ng) &= \sum_{j=1}^n(\tilde{L}_j(\xi) h_{j}) dW + \xi \cdbarst (Ng),
    \end{align*}
    where $\tilde{L}_j$s are first order differential operators. Based on these observations, we can now write \eqref{imp1} as follows,
    \begin{equation}\label{imp1.1}
    	\begin{aligned}
    		\norm{}\xi N g\norm{}^2_{s+\lambda}&\leq E (\norm{}\cdbar (\xi N g)\norm{}^2_{s}+ \norm{}\cdbarst (\xi N g)\norm{}^2_{s} + \norm{}\xi N g\norm{}^2_{s}) \\
    		& \leq E (\norm{}(\cdbar \xi) \wedge N g\norm{}^2_{s}+ \norm{}\xi (\cdbar N g)\norm{}^2_{s}+ \norm{}\sum_{j=1}^n(\tilde{L}_j(\xi) h_{j})\norm{}^2_{s}+\norm{}\xi (\cdbarst N g)\norm{}^2_{s} + \norm{}\xi N g\norm{}^2_{s}).
    	\end{aligned}
    \end{equation}
    We can estimate the first and third term in the right hand side of the last inequality as follows,
	\begin{equation}\label{imp2}
		\begin{aligned}
			\norm{}(\cdbar \xi) \wedge N g\norm{}^2_{s}+\norm{}\sum_{j=1}^n(\tilde{L}_j(\xi) h_{j})\norm{}^2_{s} &= \norm{}(\cdbar \xi)\wedge \xi_1 N g\norm{}^2_{s}+ \norm{}\sum_{j=1}^n(\tilde{L}_j(\xi) \xi_1 h_{j})\norm{}^2_{s} \\
			& \leq E \norm{}\xi\norm{}_{c}^2\norm{}\xi_1 N g\norm{}^2_{s}.
		\end{aligned}
	\end{equation}
	The last inequality holds as the $L^{\infty}$ norms of $\xi$ up to $(s+1)$th order derivatives will be bounded above by $\norm{}\xi\norm{}_c$, because of the assumption on $c$. Likewise, $\norm{}\xi N g\norm{}^2_{s} = \norm{}\xi \xi_1 N g\norm{}^2_{s}\leq E \norm{}\xi\norm{}_{c}^2 \norm{}\xi_1 Ng\norm{}^2_{s}$. Thus, it follows from \eqref{imp1.1} and \eqref{imp2} that
	\begin{equation}\label{imp3}
		\begin{aligned}
			\norm{}\xi N g\norm{}^2_{s+\lambda}& \leq E (\norm{}\xi\norm{}_{c}^2 \norm{}\xi_1 N g\norm{}^2_{s}+ \norm{}\xi (\cdbar N g)\norm{}^2_{s} + \norm{}\xi (\cdbarst N g)\norm{}^2_{s} + \norm{}\xi N g\norm{}^2_{s}) \\
			& \leq E (\norm{}\xi\norm{}_{c}^2 \norm{}\xi_1 N g\norm{}^2_{s}+ \norm{}\xi (\cdbar N g)\norm{}^2_{s} + \norm{}\xi (\cdbarst N g)\norm{}^2_{s}).
		\end{aligned}
	\end{equation}
	Now we investigate the last two terms in the right hand side of \eqref{imp3} separately. By definition of Sobolev norms of order $s \in \mathbb{N} \cup \{0\}$, we have
	\begin{align}\label{est1dbar}
		\begin{split}
			&\norm{}\xi (\cdbar N g)\norm{}^2_{s} \leq  E\sum_{|\beta|+|\gamma| \leq s} \norm{}(D^{\beta}\xi) D^{\gamma}(\cdbar N g)\norm{}^2_{0}=E \sum_{|\beta|+|\gamma| \leq s} ( (D^{\beta}\xi) D^{\gamma}(\cdbar N g), (D^{\beta}\xi) D^{\gamma}(\cdbar N g)); \\
			&\norm{}\xi (\cdbarst N g)\norm{}^2_{s} \leq E \sum_{|\beta|+|\gamma| \leq s} \norm{}(D^{\beta}\xi) D^{\gamma}(\cdbarst N g)\norm{}^2_{0}= E \sum_{|\beta|+|\gamma| \leq s} ( (D^{\beta}\xi) D^{\gamma}(\cdbarst N g), (D^{\beta}\xi) D^{\gamma}(\cdbarst N g)).
		\end{split}
	\end{align}
	where we may assume all derivatives $D^{\beta}, D^{\gamma}$ are real derivatives. To estimate $\norm{}\xi (\cdbar N g)\norm{}^2_{s}$ and $\norm{}\xi (\cdbarst N g)\norm{}^2_{s}$, we estimate $\norm{}(D^{\beta}\xi) D^{\gamma}(\cdbar N g)\norm{}^2_{0}$ and $\norm{}(D^{\beta}\xi) D^{\gamma}(\cdbarst N g)\norm{}^2_{0}$ for $|\beta|+|\gamma|\leq s$. We shall use the following lemma for estimating these terms.
	\begin{lemma}\label{lemma:estimatenormaltangent}
		Let $M,U$ be as in Claim \ref{claimA1}. Let $h$ be a $(n,1)$ form that is smooth up to the boundary and $\chi \in C_c^{\infty}(\overline{M} \cap U)$. For differential operators $P_1, P_2, \ldots, P_m \in S_{tang} \cup \{Q\}$, denoting $P=P_1P_2 \cdots P_m$, we have
		\begin{align}\label{app1lemma2eq1}
			\norm{}\chi P h\norm{}_{0}^2 \leq E \left(\sum_{|\gamma| \leq m-1}\norm{}\chi D^{\gamma} h\norm{}_{0}^2 + \sum_{|\gamma| \leq m}\norm{}\chi D^{\gamma}_{tang} h\norm{}_{0}^2 +
			\sum_{|\gamma| \leq m-1}\norm{}\chi D^{\gamma} \cdbar h\norm{}_{0}^2 + \sum_{|\gamma| \leq m-1}\norm{}\chi D^{\gamma} \cdbarst h\norm{}_{0}^2\right).
		\end{align}
		where $E$ can depend on $M,U,p$ but not on $h$. In the above, $D^{\gamma}$ runs over all the differential operators of order $|\gamma|$, $D^{\gamma}_{tang}$ runs over all the differential operators of order $|\gamma|$, consisting of tangential derivatives only. Here $P$, $D^{\gamma}$, etc., act on a form component wise.
	\end{lemma}
	\begin{proof}
		If $P_1, \ldots P_m \in S_{tang}$, then the claim is trivial as then $\norm{}\chi P h\norm{}_{0}^2$ is definitely bounded above by $\sum_{|\gamma| \leq m}\norm{}\chi D^{\gamma}_{tang} h\norm{}_{0}^2$ as we sum over all tangential vector fields. So now we assume $P_{i_0}=Q$ for some $1 \leq i_0 \leq m$. If $i_0 < m$ then we can write $P=P_1 \cdots P_{i_0-1}P_{i_0+1}\cdots P_mQ+ P_1 \cdots P_{i_0-1}[Q,P_{i_0+1}\cdots P_m]$. Let us call $P_1 \cdots P_{i_0-1}P_{i_0+1}\cdots P_mQ$ as $\tilde{P}$ and we note that $P_1 \cdots P_{i_0-1}[Q,P_{i_0+1}\cdots P_m]$ is a differential operator of order $m-1$. If $i_0=m$ then we set $\tilde{P}=P$. Hence we have
		\begin{align}\label{app1lemma2eq2}
			\norm{}\chi P h\norm{}_{0}^2 \leq (\norm{}\chi \tilde{P} h\norm{}_{0}^2 + \sum_{|\gamma| \leq m-1} \norm{}\chi D^{\gamma} h\norm{}_{0}^2).
		\end{align}
		Write $h_j$ for the components of $h$ in terms of $dW \wedge \bar{\omega}_j$s. We may express $Q(h_{j})$s in the form given by equation \eqref{normalderintermsofothers} of Lemma \ref{straubelemma2.2gen}. We can rename $P_1,\ldots, P_{i_0-1},P_{i_0+1},\ldots P_m$ as $\tilde{P}_1, \ldots ,\tilde{P}_{m-1}$ and estimate $\chi \tilde{P}_1 \cdots \tilde{P}_{m-1}$ applied to $Q(h)$ (where $Q$ is acting coefficient wise)  using \eqref{normalderintermsofothers} as follows
		\begin{align}\label{app1lemma2eq3}
			\norm{}\chi \tilde{P} h\norm{}_{0}^2 \leq E \left(\sum_{D_i \in S_{tang}}\sum_{j} \norm{}\chi \tilde{P}_1 \cdots \tilde{P}_{m-1}(D_ih_{j})\norm{}^2_{0}+ \sum_{|\gamma| \leq m-1}\norm{}\chi D^{\gamma} h\norm{}_{0}^2 +
			\sum_{|\gamma| \leq m-1}\norm{}\chi D^{\gamma} \cdbar h\norm{}_{0}^2 + \sum_{|\gamma| \leq m-1}\norm{}\chi D^{\gamma} \cdbarst h\norm{}_{0}^2\right).
		\end{align}
		Combining \eqref{app1lemma2eq2} and \eqref{app1lemma2eq3}, we obtain
		\begin{align}\label{app1lemma2eq4}
			\norm{}\chi P h\norm{}_{0}^2 \leq E \left(\sum_{D_i \in S_{tang}}\sum_{j} \norm{}\chi \tilde{P}_1 \cdots \tilde{P}_{m-1}(D_ih_{j})\norm{}^2_{0}+ \sum_{|\gamma| \leq m-1}\norm{}\chi D^{\gamma} h\norm{}_{0}^2 +
			\sum_{|\gamma| \leq m-1}\norm{}\chi D^{\gamma} \cdbar h\norm{}_{0}^2 + \sum_{|\gamma| \leq m-1}\norm{}\chi D^{\gamma} \cdbarst h\norm{}_{0}^2\right).
		\end{align}
		Again if $\tilde{P}_1, \ldots \tilde{P}_{m-1} \in S_{tang}$ then the term $\sum_{D_i \in S_{tang}}\sum_{j} \norm{}\chi \tilde{P}_1 \cdots \tilde{P}_{m-1}(D_ih_{j})\norm{}^2_{0}$ in the right hand side of \eqref{app1lemma2eq4} can clearly be bounded above by $\sum_{|\gamma| \leq m}\norm{}\chi D^{\gamma}_{tang} h\norm{}_{0}^2$, proving our lemma. Otherwise, similarly as above, we write $\tilde{P}_1 \cdots \tilde{P}_{m-1}D_i = \hat{P}D_iQ + \ \text{a lower order differential operator}$, where $\hat{P}$ is a differential operator of order $m-2$. Proceeding exactly in a similar manner as before, we obtain
		\begin{align}\label{app1lemma2eq5}
			\norm{}\chi P h\norm{}_{0}^2 \leq E (\sum_{D_i,D_k \in S_{tang}}\sum_{j} \norm{}\chi \hat{P}(D_iD_kh_{j})\norm{}^2_{0}+ \sum_{|\gamma| \leq m-1}\norm{}\chi D^{\gamma} h\norm{}_{0}^2 +
			\sum_{|\gamma| \leq m-1}\norm{}\chi D^{\gamma} \cdbar h\norm{}_{0}^2 + \sum_{|\gamma| \leq m-1}\norm{}\chi D^{\gamma} \cdbarst h\norm{}_{0}^2).
		\end{align}
		Repeating this process at most $m-2$ times, we arrive at \eqref{app1lemma2eq1}.
	\end{proof}

\bigskip
	
	In \eqref{app1lemma2eq1}, let us take $\chi=D^{\beta}\xi$, $h=\cdbar Ng$, $P=D^{\gamma}$ with $|\beta|+|\gamma| \leq s$. This gives us
	\begin{align}\label{app1final1}
		\norm{}D^{\beta}\xi D^{\gamma}(\cdbar Ng)\norm{}_{0}^2 &\leq E \left(\sum_{|\alpha| \leq |\gamma|-1}\norm{}D^{\beta}\xi D^{\alpha}(\cdbar Ng)\norm{}_{0}^2 + \sum_{|\alpha| \leq |\gamma|}\norm{}D^{\beta}\xi D^{\alpha}_{tang}(\cdbar Ng)\norm{}_{0}^2 + \sum_{|\alpha| \leq |\gamma|-1}\norm{}D^{\beta}\xi D^{\alpha} \cdbarst \cdbar Ng\norm{}_{0}^2\right).
	\end{align}
	Similarly in \eqref{app1lemma2eq1} if we take $\chi=D^{\beta}\xi$, $h=\cdbarst Ng$, $P=D^{\gamma}$ with $|\beta|+|\gamma| \leq s$, then we obtain
	\begin{align}\label{app1final2}
		\norm{}D^{\beta}\xi D^{\gamma}(\cdbarst Ng)\norm{}_{0}^2 &\leq E \left(\sum_{|\alpha| \leq |\gamma|-1}\norm{}D^{\beta}\xi D^{\alpha}(\cdbarst Ng)\norm{}_{0}^2 + \sum_{|\alpha| \leq |\gamma|}\norm{}D^{\beta}\xi D^{\alpha}_{tang}(\cdbarst Ng)\norm{}_{0}^2 + \sum_{|\alpha| \leq |\gamma|-1}\norm{}D^{\beta}\xi D^{\alpha} \cdbar\cdbarst Ng\norm{}_{0}^2\right).
	\end{align}
	We note if we sum up the first terms in the right hand sides of \eqref{app1final1} and \eqref{app1final2} together, we get
	\begin{equation}\label{app1final3}
		\begin{aligned}
		    \sum_{\substack{\beta, \gamma \\ |\beta|+|\gamma| \leq s}}\sum_{|\alpha| \leq |\gamma|-1}(\norm{}D^{\beta}\xi D^{\alpha}(\cdbar Ng)\norm{}_{0}^2 + \norm{}D^{\beta}\xi D^{\alpha}(\cdbarst Ng)\norm{}_{0}^2) &\leq 2 \sum_{\substack{\beta, \gamma \\ |\beta|+|\gamma| \leq s}}\sum_{|\alpha| \leq |\gamma|}\norm{}D^{\beta}\xi D^{\alpha}( Ng)\norm{}_{0}^2\\
		    & \leq 2 \sum_{\substack{\beta, \gamma \\ |\beta|+|\gamma| \leq s}}\sum_{|\alpha| \leq |\gamma|}\norm{}D^{\beta}\xi D^{\alpha}(\xi_1 Ng)\norm{}_{0}^2 \\
		    & \leq C_s \norm{}\xi\norm{}_{c}^2 \sum_{\substack{\gamma \\ |\gamma| \leq s}}\norm{}D^{\gamma}(\xi_1 Ng)\norm{}_{0}^2 \\
			&= C_s \norm{}\xi\norm{}_{c}^2 \norm{}\xi_1 Ng\norm{}_{s}^2,
		\end{aligned}
	\end{equation}
where the constants $C_s$ in the end can also depend on $s$ now. We have used here the fact that $\xi_1=1$ on an open neighborhood of $supp(\xi)$.
Now we estimate the last terms on the right hand sides of \eqref{app1final1} and \eqref{app1final2}, i.e., $\sum_{|\alpha| \leq |\gamma|-1}\norm{}D^{\beta}\xi D^{\alpha} \cdbarst \cdbar Ng\norm{}_{0}^2$ and $\sum_{|\alpha| \leq |\gamma|-1}\norm{}D^{\beta}\xi D^{\alpha} \cdbar\cdbarst Ng\norm{}_{0}^2$. For doing so we recall that $g= \cdbar f$ where $f$ is a smooth $(n,0)$ form. Let us denote by $id$  the identity operator here and $H$ the projection from $L^2$ integrable forms to the harmonic forms (without mentioning the exact type of the forms every time). We recall that $\cdbar N f=N\cdbar f=Ng$ (see, e.g., \cite{FollandKohndbarneumanntheorybook}).
From this observation, we see
    \begin{equation*}
    	D^{\beta}\xi D^{\alpha} \cdbarst \cdbar Ng = D^{\beta}\xi D^{\alpha} \cdbarst \cdbar^2 Nf=0,
    \end{equation*}
and
\begin{align*}
D^{\beta}\xi D^{\alpha} \cdbar\cdbarst Ng &= D^{\beta}\xi D^{\alpha} \cdbar\cdbarst \cdbar Nf = D^{\beta}\xi D^{\alpha} \cdbar(\cdbarst \cdbar + \cdbar \cdbarst) Nf \\
    	&= D^{\beta}\xi D^{\alpha} \cdbar(id-H)Nf = D^{\beta}\xi D^{\alpha} \cdbar Nf - D^{\beta}\xi D^{\alpha} \cdbar HNf \\
    	&= D^{\beta}\xi D^{\alpha} \cdbar Nf  = D^{\beta}\xi D^{\alpha} Ng.
 \end{align*}

    As we sum over $\alpha, \beta, \gamma$, we have
    \begin{equation}\label{app1final5}
    	\begin{aligned}
    		\sum_{\substack{|\beta|+|\gamma| \leq s}} \sum_{|\alpha|\leq |\gamma|-1} \norm{}D^{\beta}\xi D^{\alpha} \cdbarst \cdbar Ng\norm{}_{0}^2 +  \norm{}D^{\beta}\xi D^{\alpha} \cdbar\cdbarst Ng\norm{}_{0}^2
    		& = \sum_{\substack{|\beta|+|\gamma| \leq s}} \sum_{|\alpha|\leq |\gamma|-1}(\norm{}D^{\beta}\xi D^{\alpha} Ng\norm{}^2_{0}) \\
    		&\leq C_s\norm{}\xi\norm{}_{c}^2\norm{}\xi_1 Ng\norm{}^2_{s}.
    	\end{aligned}
    \end{equation}
The last inequality above follows similarly as in (\ref{app1final3}). Finally we estimate the terms $\sum_{|\alpha| \leq |\gamma|}\norm{}D^{\beta}\xi D^{\alpha}(\cdbar Ng)\norm{}_{0}^2$ and $\sum_{|\alpha| \leq |\gamma|}\norm{}D^{\beta}\xi D^{\alpha}(\cdbarst Ng)\norm{}_{0}^2$ in the right hand sides of \eqref{app1final1} and \eqref{app1final2} when $D^{\alpha}$ consists of only tangential vector fields. This also includes the trivial case of $|\alpha|=0$. Based on the product rule of $\cdbar$, we obtain
	\begin{align*}
		(D^{\beta}\xi) D^{\alpha}(\cdbar N g) &= (D^{\beta}\xi) \cdbar D^{\alpha}(N g)- (D^{\beta}\xi) [\cdbar, D^{\alpha}](N g) \\
		&= \cdbar ((D^{\beta}\xi) D^{\alpha}(N g)) -  (\cdbar D^{\beta}\xi) \wedge D^{\alpha}(N g) - (D^{\beta}\xi) [\cdbar, D^{\alpha}](N g).
	\end{align*}
	We now estimate $\norm{}D^{\beta}\xi D^{\alpha}(\cdbar Ng)\norm{}_{0}^2=( (D^{\beta}\xi) D^{\alpha}(\cdbar N g), (D^{\beta}\xi) D^{\alpha}(\cdbar N g))$ by using the splitting of terms from this last equation in the second argument of this inner product and using Cauchy-Schwarz,
	\begin{equation}\label{est1dbar1}
		\begin{aligned}
			&|( (D^{\beta}\xi) D^{\alpha}(\cdbar N g), (D^{\beta}\xi) [\cdbar, D^{\alpha}](N g))| \leq sc \norm{}(D^{\beta}\xi) D^{\alpha}(\cdbar N g)\norm{}^2_{0} + lc \norm{}(D^{\beta}\xi) [\cdbar, D^{\alpha}](N g)\norm{}^2_{0}; \\
			&|( (D^{\beta}\xi) D^{\alpha}(\cdbar N g), (\cdbar D^{\beta}\xi) \wedge D^{\alpha}(N g))| \leq sc \norm{}(D^{\beta}\xi) D^{\alpha}(\cdbar N g)\norm{}^2_{0} + lc \norm{}(\cdbar D^{\beta}\xi) \wedge D^{\alpha}(N g)\norm{}^2_{0};\\
			&\text{and} \ \norm{}(\cdbar D^{\beta}\xi) \wedge D^{\alpha}(N g)\norm{}^2_{0} \leq \norm{}(D^{\beta} \cdbar \xi) \wedge D^{\alpha}(N g)\norm{}^2_{0} + \norm{}([\cdbar, D^{\beta}]\xi) \wedge D^{\alpha}(N g)\norm{}^2_{0}.
		\end{aligned}
	\end{equation}
Here and in the following, ``sc" and ``lc" stand for some appropriate small and large constants, respectively. Note that to get \eqref{est1dbar1}, we did not really use the fact that $D^{\alpha}$ is tangential. Since our intention is to estimate $\norm{}(D^{\beta}\xi) D^{\alpha}(\cdbar N g)\norm{}^2_{0}$ explicitly, the small constants that we get in the last inequalities above can be absorbed in $\norm{}(D^{\beta}\xi) D^{\alpha}(\cdbar N g)\norm{}^2_{0}$, in a standard way. We next estimate the term $( (D^{\beta}\xi) D^{\alpha}(\cdbar N g), \cdbar ((D^{\beta}\xi) D^{\alpha}(N g)))$. For that we recall $D^{\alpha},D^{\beta}$ are real derivatives and $\xi$ is a real cutoff function. Hence we can write:
	\begin{align}\label{est2dbar1}
		\begin{split}
			( (D^{\beta}\xi) D^{\alpha}(\cdbar N g), \cdbar ((D^{\beta}\xi) D^{\alpha}(N g))) &= ( D^{\alpha}(\cdbar N g), (D^{\beta}\xi) \cdbar ((D^{\beta}\xi) D^{\alpha}(N g)))\\
			&=( D^{\alpha}(\cdbar N g), \cdbar ((D^{\beta}\xi)^2 D^{\alpha}(N g))) - ( D^{\alpha}(\cdbar N g), \cdbar(D^{\beta}\xi) \wedge ((D^{\beta}\xi) D^{\alpha}(N g))).
		\end{split}
	\end{align}
	Now the $\cdbar$ on the second term of the first inner product in the right hand side of \eqref{est2dbar1} can be shifted to first term of this inner product as $\cdbarst$. This is because we know that $\cdbar N g$ is in the domain of $\cdbarst$. Since $D^{\alpha}$ is made up of purely tangential vector fields (acting coefficient wise on a form), we see $D^{\alpha} \cdbar N g$ is in the domain of $\cdbarst$ as well. Thus we rewrite \eqref{est2dbar1} as:
    \begin{equation}\label{est2dbar2}
    	\begin{split}
    		( (D^{\beta}\xi) D^{\alpha}(\cdbar N g), \cdbar ((D^{\beta}\xi) D^{\alpha}(N g))) &=( \cdbarst D^{\alpha}(\cdbar N g), (D^{\beta}\xi)^2 D^{\alpha}(N g)) - ( D^{\alpha}(\cdbar N g), \cdbar(D^{\beta}\xi) \wedge ((D^{\beta}\xi) D^{\alpha}(N g))) \\
    		&= ( (D^{\beta}\xi) \cdbarst D^{\alpha}(\cdbar N g), (D^{\beta}\xi) D^{\alpha}(N g)) - ( D^{\alpha}(\cdbar N g), \cdbar(D^{\beta}\xi) \wedge ((D^{\beta}\xi) D^{\alpha}(N g))) \\
    		&= ( (D^{\beta}\xi) D^{\alpha} \cdbarst (\cdbar N g), (D^{\beta}\xi) D^{\alpha}(N g)) - ( D^{\alpha}(\cdbar N g), (D^{\beta}\cdbar\xi) \wedge ((D^{\beta}\xi) D^{\alpha}(N g))) \\
    		&+ ( (D^{\beta}\xi) [\cdbarst, D^{\alpha}](\cdbar N g), (D^{\beta}\xi) D^{\alpha}(N g)) - ( D^{\alpha}(\cdbar N g), ([\cdbar, D^{\beta}]\xi) \wedge ((D^{\beta}\xi) D^{\alpha}(N g))) \\
    		&= ( (D^{\beta}\xi) D^{\alpha} \cdbarst (\cdbar N g), (D^{\beta}\xi) D^{\alpha}(N g)) - ( (D^{\beta}\xi) D^{\alpha}(\cdbar N g), (D^{\beta}\cdbar\xi) \wedge (D^{\alpha}(N g))) \\
    		&+ ( (D^{\beta}\xi) [\cdbarst, D^{\alpha}](\cdbar N g), (D^{\beta}\xi) D^{\alpha}(N g)) - ( (D^{\beta}\xi) D^{\alpha}(\cdbar N g), ([\cdbar, D^{\beta}]\xi) \wedge D^{\alpha}(N g) )
    	\end{split}
    \end{equation}
    Similarly, we rewrite $\norm{}D^{\beta}\xi D^{\alpha}(\cdbarst Ng)\norm{}_{0}^2=( (D^{\beta}\xi) D^{\alpha}(\cdbarst N g), (D^{\beta}\xi) D^{\alpha}(\cdbarst N g))$ in the following way:
    \begin{equation}\label{est2dbar3}
    	\begin{aligned}
    		( (D^{\beta}\xi) D^{\alpha}(\cdbarst N g), (D^{\beta}\xi) D^{\alpha}(\cdbarst N g)) &=( (D^{\beta}\xi) D^{\alpha}(\cdbarst N g), (D^{\beta}\xi) \cdbarst(D^{\alpha} N g)) + ( (D^{\beta}\xi) D^{\alpha}(\cdbarst N g), (D^{\beta}\xi) ([D^{\alpha},\cdbarst] N g)) \\
    		&=( (D^{\beta}\xi)^2 D^{\alpha}(\cdbarst N g), \cdbarst(D^{\alpha} N g)) + ( (D^{\beta}\xi) D^{\alpha}(\cdbarst N g), (D^{\beta}\xi) ([D^{\alpha},\cdbarst] N g))\\
    		&=( \cdbar((D^{\beta}\xi)^2 D^{\alpha}(\cdbarst N g)), (D^{\alpha} N g)) + ( (D^{\beta}\xi) D^{\alpha}(\cdbarst N g), (D^{\beta}\xi) ([D^{\alpha},\cdbarst] N g))\\
    		&=2 ( D^{\beta}\xi\cdbar(D^{\beta}\xi) \wedge D^{\alpha}(\cdbarst N g), (D^{\alpha} N g)) + ( (D^{\beta}\xi)^2 \cdbar(D^{\alpha}(\cdbarst N g)), (D^{\alpha} N g)) \\
    		&+ ( (D^{\beta}\xi) D^{\alpha}(\cdbarst N g), (D^{\beta}\xi) ([D^{\alpha},\cdbarst] N g)) \\
    		&=2 ( D^{\beta}\xi\cdbar(D^{\beta}\xi) \wedge D^{\alpha}(\cdbarst N g), (D^{\alpha} N g)) + ( (D^{\beta}\xi)^2 (D^{\alpha}\cdbar(\cdbarst N g), (D^{\alpha} N g)) \\
    		&+ ( (D^{\beta}\xi)^2 [\cdbar, D^{\alpha}]\cdbarst N g, (D^{\alpha} N g)) + ( (D^{\beta}\xi) D^{\alpha}(\cdbarst N g), (D^{\beta}\xi) ([D^{\alpha},\cdbarst] N g)) \\
    		&=2 ( D^{\beta}\xi\cdbar(D^{\beta}\xi) \wedge D^{\alpha}(\cdbarst N g), (D^{\alpha} N g)) + ( (D^{\beta}\xi) D^{\alpha}\cdbar(\cdbarst N g), D^{\beta}\xi D^{\alpha} N g ) \\
    		&+ ( D^{\beta}\xi [\cdbar, D^{\alpha}]\cdbarst N g, D^{\beta}\xi D^{\alpha} N g ) + ( (D^{\beta}\xi) D^{\alpha}(\cdbarst N g), (D^{\beta}\xi) ([D^{\alpha},\cdbarst] N g)).
    	\end{aligned}
    \end{equation}
    First we note that when we are estimating $\norm{}D^{\beta}\xi D^{\alpha}(\cdbar Ng)\norm{}_{0}^2 + \norm{}D^{\beta}\xi D^{\alpha}(\cdbarst Ng)\norm{}_{0}^2$ together, the term \\ $( (D^{\beta}\xi) D^{\alpha} \cdbarst (\cdbar N g), (D^{\beta}\xi) D^{\alpha}(N g))$ in the right hand side of \eqref{est2dbar2} and $( (D^{\beta}\xi) D^{\alpha}\cdbar(\cdbarst N g), D^{\beta}\xi D^{\alpha} N g )$ in the right hand side of \eqref{est2dbar3} can be added together to give us $( (D^{\beta}\xi) D^{\alpha}(id-H) g, D^{\beta}\xi D^{\alpha} N g )$. Since we assumed in the statement of Claim \ref{claimA1} that $g$ is in $\text{Range}(\cdbar)$, it is orthogonal to the harmonic $(n, 1)$ forms. So we get $(id-H)g=g$. This shows that $( (D^{\beta}\xi) D^{\alpha}(id-H) g, D^{\beta}\xi D^{\alpha} N g )=( (D^{\beta}\xi) D^{\alpha}g, D^{\beta}\xi D^{\alpha} N g )$. Let us then estimate this term and few more terms in the right hand sides of \eqref{est2dbar2} and \eqref{est2dbar3}:
    \begin{equation}\label{est2dbar4}
    	\begin{aligned}
    		|( (D^{\beta}\xi) D^{\alpha} g, D^{\beta}\xi D^{\alpha} N g )| &\leq \norm{}(D^{\beta}\xi) D^{\alpha} g\norm{}^2_{0} + \norm{}(D^{\beta}\xi) D^{\alpha} Ng\norm{}^2_{0}, \\
    		|( (D^{\beta}\xi) D^{\alpha}(\cdbar N g), (D^{\beta}\cdbar\xi) \wedge (D^{\alpha}(N g)))| &\leq sc \norm{}(D^{\beta}\xi) D^{\alpha}(\cdbar N g)\norm{}^2_{0} + lc \norm{}(D^{\beta}\cdbar\xi) \wedge (D^{\alpha}(N g))\norm{}^2_{0}, \\
    		|( (D^{\beta}\xi) [\cdbarst, D^{\alpha}](\cdbar N g), (D^{\beta}\xi) D^{\alpha}(N g))| & \leq sc \norm{}(D^{\beta}\xi) [\cdbarst,D^{\alpha}](\cdbar N g)\norm{}^2_{0} + lc \norm{}(D^{\beta}\xi) D^{\alpha}(N g)\norm{}^2_{0}, \\
    		|( (D^{\beta}\xi) D^{\alpha}(\cdbar N g), ([\cdbar, D^{\beta}]\xi) \wedge D^{\alpha}(N g) )| &\leq sc \norm{}(D^{\beta}\xi) D^{\alpha}(\cdbar N g)\norm{}^2_{0} + lc \norm{}([\cdbar, D^{\beta}]\xi) \wedge D^{\alpha}(N g)\norm{}^2_{0}, \\
    		|( D^{\beta}\xi [\cdbar, D^{\alpha}]\cdbarst N g, D^{\beta}\xi D^{\alpha} N g )| &\leq sc\norm{} D^{\beta}\xi [\cdbar, D^{\alpha}]\cdbarst N g\norm{}^2_{0} + lc \norm{}D^{\beta}\xi D^{\alpha} N g\norm{}^2_{0},\\
    		|( (D^{\beta}\xi) D^{\alpha}(\cdbarst N g), (D^{\beta}\xi) ([D^{\alpha},\cdbarst] N g))| &\leq sc \norm{}(D^{\beta}\xi) D^{\alpha}(\cdbarst N g)\norm{}^2_{0} + lc \norm{}(D^{\beta}\xi) ([D^{\alpha},\cdbarst] N g)\norm{}^2_{0}.
    	\end{aligned}
    \end{equation}
    The differential operators $[\cdbarst, D^{\alpha}]$, $[\cdbar, D^{\beta}]$, $[\cdbar, D^{\alpha}]$ are of order $|\alpha|, |\beta|$, and $|\alpha|$ respectively. The terms with small constants in the right hand sides of \eqref{est2dbar4} above, can be absorbed using the standard trick as we sum over $\alpha, \beta, \gamma$. Finally we estimate the term $2 ( D^{\beta}\xi\cdbar(D^{\beta}\xi) \wedge D^{\alpha}(\cdbarst N g), (D^{\alpha} N g))$ coming from \eqref{est2dbar3}. Since $g$ is a $(n, 1)$ form, let us write $D^{\beta}\xi D^{\alpha}\cdbarst Ng = A dW$, $D^{\alpha} Ng = \sum_{j=1}^n B_{j} \bar{\omega}_j \wedge dW$, and $\cdbar D^{\beta} \xi = \sum_{j=1}^n \bar{L}_j(D^{\beta}\xi) \bar{\omega}_j$ where $A, B_{j}$s are functions smooth up to $\overline{M}$. Note
    \begin{align*}
    	( D^{\beta}\xi\cdbar(D^{\beta}\xi) \wedge D^{\alpha}(\cdbarst N g), (D^{\alpha} N g)) &= ( \cdbar(D^{\beta}\xi) \wedge D^{\beta}\xi D^{\alpha}(\cdbarst N g), (D^{\alpha} N g)) \\
    	&= ( \sum_{j=1}^n A\bar{L}_j(D^{\beta}\xi) \bar{\omega}_j \wedge dW,  \sum_{j=1}^n B_{j} \bar{\omega}_j \wedge dW ) \\
    	&=\sum_{j=1}^n ( A\bar{L}_j(D^{\beta}\xi), B_{j} ) = \sum_{j=1}^n ( A, L_j(D^{\beta}\xi)B_{j} )\\
    \end{align*}
    By applying Cauchy-Schwarz, we obtain
    \begin{equation}\label{est2dbar5}
    	|( D^{\beta}\xi\cdbar(D^{\beta}\xi) \wedge D^{\alpha}(\cdbarst N g), (D^{\alpha} N g))| \leq sc\norm{}D^{\beta}\xi D^{\alpha}\cdbarst Ng\norm{}^2_{0} + lc\norm{}L(D^{\beta}\xi) \cdot D^{\alpha}Ng\norm{}^2_{0} 
    \end{equation}
    where by $\norm{}L(D^{\beta}\xi) \cdot D^{\alpha}Ng\norm{}^2_{0}$ we mean $\sum_{j=1}^n\norm{}L_j(D^{\beta}\xi)B_{j}\norm{}^2_{0}$ 
    Combining the estimates from \eqref{est1dbar1}, \eqref{est2dbar2}, \eqref{est2dbar3}, \eqref{est2dbar4}, \eqref{est2dbar5}, we obtain
  	\begin{equation}\label{app1final4}
  	\begin{aligned}
  		 \sum_{\substack{\beta, \gamma \\ |\beta|+|\gamma| \leq s}}\sum_{|\alpha| \leq |\gamma|}\norm{}(D^{\beta}\xi) D^{\alpha}(\cdbar N g)\norm{}^2_{0} + \norm{}(D^{\beta}\xi) D^{\alpha}(\cdbarst N g)\norm{}^2_{0}
  		 \leq & \sum_{\substack{\beta, \gamma \\ |\beta|+|\gamma| \leq s}}\sum_{|\alpha| \leq |\gamma|}E (\norm{}(D^{\beta}\xi) [\cdbar, D^{\alpha}](N g)\norm{}^2_{0} \\
  		 &+ \norm{}(D^{\beta}\cdbar \xi) \wedge D^{\alpha}(N g)\norm{}^2_{0}+ \norm{}([\cdbar, D^{\beta}] \xi) \wedge D^{\alpha}(N g)\norm{}^2_{0} \\
  		 &+ \norm{}(\cdbar D^{\beta} \xi) \wedge D^{\alpha}(N g)\norm{}^2_{0} + \norm{}(D^{\beta}\xi) [D^{\alpha}, \cdbarst](N g)\norm{}^2_{0}  \\
  		& + \norm{}(D^{\beta}\xi) D^{\alpha}g\norm{}^2_{0} + \norm{}(D^{\beta} \xi) D^{\alpha}(N g)\norm{}^2_{0} \\
  		& + \norm{}L(D^{\beta}\xi) \cdot D^{\alpha}Ng\norm{}^2_{0} \\ 
  		 \leq & \, C_s \norm{}\xi\norm{}^2_{c}(\norm{}\xi_1 Ng\norm{}^2_{s} + \norm{}\xi_1 g\norm{}^2_{s}).
  	 \end{aligned}
  \end{equation}
The last inequality above follows similarly as in (\ref{app1final3}). We combine the estimates from  \eqref{est1dbar}, \eqref{app1final1}, \eqref{app1final2},  \eqref{app1final3}, \eqref{app1final5} and \eqref{app1final4} to get:
	\begin{equation}\label{imp4}
		\begin{aligned}
			\norm{}\xi (\cdbar N g)\norm{}^2_{s} + \norm{}\xi (\cdbarst N g)\norm{}^2_{s} &\leq C_s \norm{}\xi\norm{}_{c}^2(\norm{}\xi_1 N g\norm{}^2_{s} + \norm{}\xi_1 g\norm{}^2_{s}).
		\end{aligned}
	\end{equation}
	Hence, in the end, from \eqref{imp3} and \eqref{imp4} we obtain the desired estimate:
	\begin{align*}
		\norm{}\xi N g\norm{}^2_{s+\lambda}&\leq C_s \norm{}\xi\norm{}_{c}^2(\norm{}\xi_1 N g\norm{}^2_{s} + \norm{}\xi_1 g\norm{}^2_{s}).
	\end{align*}
 \end{proof}


\section{An iteration scheme}
\label{appendixB}
For the convenience of the reader, we provide in this appendix a slightly modified version of an iteration scheme due to Boas (see Appendix A in \cite{BoasiterationpaperSzegoregdom}). The result of this scheme was used in our proof of Claim \ref{claim21}.
Suppose $\Omega$ is a smoothly bounded domain in $\mathbb{R}^N$ and $L$ a linear operator $L^2(\Omega) \to L^2(\Omega)$. Here, as before, $L^2(\Omega)$ could denote the space $L_{p,q}^2(\Omega)$ of square-integrable $(p,q)$ form on $\Omega,$ although we have omitted the $(p,q)$ in the notation. Let $R(\cdot)$ be a norm on $L^2(\Omega)$ that satisfies $R(hf) \leq \norm{}h\norm{}_{c} R(f)$ for all $h \in C^{\infty}(\bar{\Omega})$. Here $c>0$ is some constant independent of $f$ and $h$. Let $m>0$ and $G$ be a norm that dominates $\norm{} \cdot \norm{}^{\dagger}_{-m}$, $U$ an open set in $\mathbb{R}^N$ such that $\Omega \cap U \neq \emptyset$, and $T$ an operator from $L^2(\Omega)$ to itself. Assume there exists $m'>0$ such that $L$ satisfies
\begin{align}\label{SobolevboundedL}
	\norm{}L f\norm{}^{\dagger}_{-m'} \leq E 	\norm{}f\norm{}^{\dagger}_{-m}.
\end{align}
for all $f \in L^2(\Omega)$ and some $E$ independent of $f$. We also assume, for a fixed $r>0$, that $c>0$ is large enough such that $\norm{}hf\norm{}_{r} \leq \norm{}h\norm{}_{c} \norm{}f\norm{}_{r}$ and $\norm{}hf\norm{}^{\dagger}_{-m'} \leq \norm{}h\norm{}_{c} \norm{}f\norm{}^{\dagger}_{-m'}$ for all $h \in C^{\infty}(\overline{\Omega})$ and $f \in L^2(\Omega)$.
\begin{lemma}\label{iterationschemeappend}
	Let $L, R, G$ be as above. Fix $r, m, \epsilon \in \mathbb{R}_{> 0}$ such that $\epsilon \leq r$. Let $m'$ be as in \eqref{SobolevboundedL} and $c>\max\{r/3,(r-\epsilon +m')/\epsilon\}$ a constant that satisfies the conditions above. Assume there exists a constant $E >0$ such that for all $ \xi, \xi_1 \in C^{\infty}_c(U)$ with $\xi_1 \equiv 1$ in a neighborhood of $supp(\xi)$, we have:
	\begin{align}\label{hypothesisestimate}
		\norm{}\xi L f\norm{}_{r} \leq E d^{-c} \norm{}\xi\norm{}_{c} \big(R(\xi_1 T f) + \norm{}\xi_1 L f\norm{}_{r - \epsilon} \big)
	\end{align}
	for all $f$ in a subspace of $L^2(\Omega)$, where $d=dist(supp(\xi), supp(1-\xi_1)) \leq 1$. Then there is a new constant $E >0$ (independent of $\xi, \xi_1$) such that:
	\begin{equation}\label{iteratedestimate}
		\begin{aligned}
			\norm{}\xi L f\norm{}_{r} \leq E d^{-3c-2c^2} \norm{}\xi\norm{}_{c} \big(R(\xi_1 T f) + G(f) \big)
		\end{aligned}
	\end{equation}
	for $f$ in the same subspace of $L^2(\Omega)$ where $\norm{}\xi_1 L f\norm{}_{r} < \infty$.
\end{lemma}
\begin{proof}[Proof of Lemma \ref{iterationschemeappend}.]
	
	Let us fix the cut-off functions $\xi$ and $\xi_1$ for which we desire to prove this lemma and $m \in \mathbb{N}$. Let us also fix a non-negative function $\phi \in C^{\infty}(\mathbb{R}^N)$, supported in the unit ball and with integral equal to one. For each positive integer $j$, set \begin{align*}
		\phi_j(x) \coloneqq (2^{j+2}d^{-1})^N \phi(2^{j+2}d^{-1}x).
	\end{align*}
	Denote by $\eta_j$ the convolution $\phi_j * \chi_{V_j}$, where $V_j \coloneqq \{x \in U~|~{dist}(x, supp(1-\xi_1)) \geq 2^{-j}d\}$.
	By standard properties of convolution, we have that $\eta_j \in C^{\infty}_c(U)$ and $supp(\eta_j) \subset \overline{supp(\phi_j) + supp(\chi_{V_j})}$.
	One can verify that $\eta_1 \equiv 1$ in a neighborhood of $supp(\xi)$ and each $\eta_{j+1}$ is identically equal to 1 in a neighborhood of $supp(\eta_j)$. The following estimate is immediate, for large $j$:
	\begin{align}\label{etajbound}
	\norm{}\eta_j\norm{}_{c}\leq E 2^{jc}d^{-c}
	\end{align}
	where the constant $E$ can depend on the derivatives of $\phi$ and $c$. Let us denote  $d_j \coloneqq dist(supp(\eta_j), supp(1-\eta_{j+1}))$. By definition of $\eta_j$s and the aforementioned property about the support of convolution, one notes that
\begin{equation*}
dist(supp(\eta_j), supp(1-\xi_1)) \geq d/2^j-d/2^{j+2}.
\end{equation*}
Also, if $x \in \Omega$ is a point such that $dist(x, supp(1-\xi_1)) \geq d/2^{j+1}+d/2^{j+3}$, then we have $\eta_{j+1}(x)=1$. Hence $1-\eta_{j+1} \equiv 0$ on $\{x \in \Omega: dist(x, supp(1-\xi_1)) \geq d/2^{j+1}+d/2^{j+3}\}$. This proves $d_j \geq  d/2^j-d/2^{j+1}-d/2^{j+2}-d/2^{j+3}=d/2^{j+3}$ which implies $d_j^{-c} \leq 8^c 2^{jc}d^{-c}$. 
Hence in the hypothesis \eqref{hypothesisestimate}, if we replace $\xi$ by $\eta_j$ and $\xi_1$ by $\eta_{j+1}$, we get :
	\begin{align}\label{apoint5boas}
		\norm{}\eta_j L f\norm{}_{r} \leq E d_j^{-c} \norm{}\eta_j\norm{}_{c} \big(R(\eta_{j+1} T f) + \norm{}\eta_{j+1} L f\norm{}_{r - \epsilon} \big) \leq E d^{-2c} 2^{2jc} \big(R(\eta_{j+1} T f) + \norm{}\eta_{j+1} L f\norm{}_{r - \epsilon} \big)
	\end{align}
	where we have used \eqref{etajbound} in the last inequality as well. We shall need the following interpolation result for Sobolev spaces.
	\begin{claim}\label{sobolevinterpolationthm}
		With $r,\epsilon,m' \in \mathbb{R}_{>0}$ fixed, if $\kappa \in (0,\infty)$, then for $\varphi \in L^2(\Omega)$
		\begin{align}\label{sobolevinterpolation}
			\norm{}\varphi\norm{}_{r-\epsilon} \leq E(\kappa \norm{}\varphi\norm{}_{r} + \kappa^{-l}\norm{}\varphi\norm{}^{\dagger}_{-m'})
		\end{align}
		where the exponent $l$ of $\kappa$ can be chosen as $(r-\epsilon +m')/\epsilon$.
	\end{claim}
	\begin{proof}
Note that we may assume that $\varphi \in H_r(\Omega)$, otherwise the conclusion is trivial.
First we prove a slightly modified version of the claim by assuming $\Omega = \mathbb{R}^N$. 
     We first show that if $w<v<u$ are real numbers, then $t^v \leq \kappa t^u + \kappa^{(v-w)/(v-u)} t^w$ for all $t \geq 1$. We see the last statement is equivalent to $t^{v-w} \leq \kappa t^{u-w} + \kappa^{(v-w)/(v-u)} = (\kappa t^{u-v})t^{v-w} + \kappa^{(v-w)/(v-u)}$. But for the latter, we note if $t \geq \kappa^{1/(v-u)}$ then our claim follows as $t^{v-w} \leq (\kappa t^{u-v})t^{v-w}$. For $t < \kappa^{1/(v-u)}$, $t^{v-w} < \kappa^{(v-w)/(v-u)}$. Once we have this observation, we take $-m'<0 \leq r-\epsilon < r$ as $w,v,u$ respectively, and take  $t=1+|\zeta|^2$, to get
		\begin{align*}
			(1+|\zeta|^2)^{r-\epsilon} \leq \kappa (1+|\zeta|^2)^r + \kappa^{-(r-\epsilon +m')/\epsilon} (1+|\zeta|^2)^{-m'}
		\end{align*}
		Choosing $l=(r-\epsilon +m')/\epsilon$, by definition of the Sobolev norm,
		\begin{align}\label{interpolationforRn}
			\norm{}\varphi\norm{}^{2}_{r-\epsilon, \mathbb{R}^N}=\int |\hat{\varphi}(\zeta)|^2 (1+|\zeta|^2)^{r-\epsilon} d \zeta \leq \kappa \int |\hat{\varphi}(\zeta)|^2 (1+|\zeta|^2)^{r} d \zeta + \kappa^{-l} \int |\hat{\varphi}(\zeta)|^2 (1+|\zeta|^2)^{-m'} d \zeta
			=\kappa \norm{}\varphi\norm{}^{2}_{r, \mathbb{R}^N} + \kappa^{-l} \norm{}\varphi\norm{}^2_{-m', \mathbb{R}^N}.
		\end{align}
Next we discuss the general case where $\Omega$ is a smoothly bounded domain in $\mathbb{R}^N$. Since our domain has smooth boundary, according to the `universal extension theorem' of Rychkov \cite[Theorem 4.1(b)]{RYCHKOV_1999} (see also \cite[Proposition~2.14]{zimingsosbolevnegpaper2022}) there exists a linear, bounded extension operator $\mathcal{E} : H_s(\Omega) \to H_s(\mathbb{R}^N)$, independent of the Sobolev index such that $\norm{}\mathcal{E} u\norm{}_{s, \mathbb{R}^N} \leq E \norm{}u\norm{}_{s}$ for all $s \in \mathbb{R}$,  where the constant $E$ may depend on $s$ but is independent of $u$. Now if we apply the extension operator on our $\varphi \in C^{\infty}(\bar{\Omega})$, we note that the support of $\mathcal{E}\varphi$ may not be compact anymore, but since we have the inequality \eqref{interpolationforRn} for any function $\varphi \in H_r(\mathbb{R}^n)$, we obtain
		\begin{align*}
			\norm{}\mathcal{E} \varphi\norm{}^{2}_{r-\epsilon, \mathbb{R}^N} \leq \kappa \norm{}\mathcal{E}\varphi\norm{}^{2}_{r, \mathbb{R}^N} + \kappa^{-l} \norm{}\mathcal{E} \varphi\norm{}^{2}_{-m', \mathbb{R}^N}.
		\end{align*}
		By boundedness of the universal extension operator $\mathcal{E}$ and definition of the Sobolev norm on a domain, the last inequality gives us
		 		\begin{align}\label{extensioninterpolation}
		 	\norm{}\varphi\norm{}^{2}_{r-\epsilon} \leq \norm{}\mathcal{E} \varphi\norm{}^{2}_{r-\epsilon, \mathbb{R}^N} \leq \kappa \norm{}\mathcal{E}\varphi\norm{}^{2}_{r, \mathbb{R}^N} + \kappa^{-l} \norm{}\mathcal{E} \varphi\norm{}^{2}_{-m', \mathbb{R}^N} \leq  E (\kappa \norm{}\varphi\norm{}^{2}_{r} + \kappa^{-l} \norm{}\varphi\norm{}^{2}_{-m'}).
		 \end{align}
		 Here the constant $E$ is independent of the choice of $\varphi$. Using the fact that $\norm{}\varphi\norm{}_{-m'} \leq \norm{}\varphi\norm{}^{\dagger}_{-m'}$  and replacing $\kappa$ by $\kappa^2$ in (\ref{extensioninterpolation}), we get

\begin{align*}
	\norm{}\varphi\norm{}^{2}_{r-\epsilon}  \leq  E (\kappa^2 \norm{}\varphi\norm{}^{2}_{r} + \kappa^{-2l} (\norm{}\varphi\norm{}^{\dagger}_{-m'})^{2}).
\end{align*}
Since for $a, b \geq 0,$ it holds that $a^2+b^2 \leq (a+b)^2$, (\ref{sobolevinterpolation}) follows immediately.

	\end{proof}
	\begin{remark}\label{Sobolevinterpolmanifold}
		\normalfont
		Let $M$ be a smoothly bounded, precompact domain $M$ in a complex Hermitian manifold $(M', g')$ and $h$ be a smooth $(p,q)$ form on $\overline{M}$. Claim \ref{sobolevinterpolationthm} still holds with $\varphi$ replaced by $h$ and $\Omega$ replaced by $M$. To see that we recall the definition of Sobolev norms from the beginning of Section \ref{genhuangli}. Let $\{U_{\gamma}\}_{\gamma \in \Gamma}$ be a locally finite covering of $M'$ by charts with coordinate mappings $\psi_{\gamma}: U_{\gamma} \to \mathbb{C}^n$, and let $\{\xi_{\gamma}\}$ be a partition of unity subordinate to $\{U_{\gamma}\}$. Fix any $U_{\gamma}$ with $U_{\gamma} \cap M \neq \emptyset.$ Let $\{\omega_{\gamma}^i\}_{i \in I}$ be a point wise orthonormal basis of $(p,q)$ forms with respect to the metric $g'$ over $U_{\gamma}$. If $h=\sum_{i \in I}h^{\gamma}_i \omega_{\gamma}^i$ in $U_{\gamma}$, we first apply \eqref{extensioninterpolation} to $\varphi = (\xi_{\gamma}h^{\gamma}_i)\circ \psi_{\gamma}^{-1}$ on $\Omega=\psi_{\gamma}(U_{\gamma})$ (where if $U_{\gamma} \not\subset M$ then we take $\varphi$ to be a Rychkov extension of $(\xi_{\gamma}h^{\gamma}_i)\circ \psi_{\gamma}^{-1}$). By noting that each $\norm{}(\xi_{\gamma}h^{\gamma}_i)\circ \psi_{\gamma}^{-1}\norm{}_{-m'}$ is bounded by some constant multiple of $\norm{}h\norm{}_{-m'}$, by summing over $\gamma, i$ (and making the constant $E$ large enough), we have:
		\begin{align*}
			\norm{}h\norm{}^2_{r-\epsilon, M} \leq E(\kappa \norm{}h\norm{}^2_{r, M} + \kappa^{-l}\norm{}h\norm{}^2_{-m', M}).
			\end{align*}
The remaining argument is the same as above.
\end{remark}

\bigskip

Using Claim \ref{sobolevinterpolationthm}, for any $\kappa \in (0,1)$ along with our assumption on $c$ (which yields $c>l$) and \eqref{apoint5boas} we get that
	\begin{align}\label{apoint6boas}
		\norm{}\eta_j L f\norm{}_{r} &\leq E d^{-2c} 2^{2jc}  \big(R(\eta_{j+1} T f) +  \kappa \norm{}\eta_{j+1} L f\norm{}_{r} + \kappa^{-c}\norm{}\eta_{j+1} L f\norm{}^{\dagger}_{-m'}\big).
	\end{align}
	We will make a special choice of $\kappa \in (0,1)$ later. From the assumption on $c$, we obtain that
	
	$$\norm{}\eta_{j+1} L f\norm{}^{\dagger}_{-m'} \leq \norm{}\eta_{j+1}\norm{}_{c} \norm{}L f\norm{}^{\dagger}_{-m'} \leq E 2^{jc+c}d^{-c} \norm{}L f\norm{}^{\dagger}_{-m'} \leq E 2^{jc}d^{-c} \norm{}f\norm{}^{\dagger}_{-m}.$$
Since $\xi_1 \equiv 1$ in a neighborhood of $supp(\eta_{j+1})$, by the assumption on the norm $R$ we conclude that
	\begin{equation*}
		R(\eta_{j+1}Lf) = R(\eta_{j+1}\xi_1Lf) \leq\norm{}\eta_{j+1}\norm{}_{c} R(\xi_1Lf) \leq E 2^{jc+c}d^{-c} R(\xi_1Lf).
	\end{equation*}
	Also, since $G(f)$ dominates $\norm{}f\norm{}^{\dagger}_{-m}$,  \eqref{apoint6boas} becomes
	\begin{equation}\label{apoint8boas}
		\begin{aligned}
			\norm{}\eta_j L f\norm{}_{r} & \leq E d^{-3c} 2^{3jc} \kappa^{-c} \big(R(\xi_1 T f) + G(f)\big) + E d^{-2c} 2^{2jc}  \kappa \norm{}\eta_{j+1} L f\norm{}_{r}.
		\end{aligned}
	\end{equation}
	Here we have also used that $\kappa \in (0,1)$. We note that $E$ depends on $L,m,$ and $c$ but is independent of $\kappa, j$, and $f$.  Now, we choose $\kappa$ to be :
	\begin{align*}
		\kappa \coloneqq \frac{1}{2E}d^{2c}2^{-2jc}2^{-3c-2c^2}
	\end{align*}
	Obviously $\kappa \in (0,1)$ for $j$ large. Let us also write $N(j) \coloneqq 2^{-j(3c+2c^2)} \norm{}\eta_{j} L f\norm{}_{r}$. With these definitions, \eqref{apoint8boas} becomes:
	\begin{equation}\label{apoint9boas}
		\begin{aligned}
			N(j) \leq \frac{1}{2}N(j+1) + E^{c+1}d^{-3c-2c^2}2^{c+3c^2+2c^3} \big(R(\xi_1 T f) + G(f)\big).
		\end{aligned}
	\end{equation}
	Because of our assumption on $c$ and the fact that $\xi_1 \equiv 1$ in a neighborhood of $supp(\eta_j)$, we have as $j \to \infty$:
	\begin{align*}
		N(j) \coloneqq \frac{\norm{}\eta_{j} L f\norm{}_{r}}{2^{j(3c+2c^2)}}=\frac{\norm{}\eta_{j} \xi_1 L f\norm{}_{r}}{2^{j(3c+2c^2)}} \leq E \frac{2^{jr}d^{-r}}{2^{j(3c+2c^2)}}\norm{}\xi_1 L f\norm{}_{r} \to 0
	\end{align*}
	Here we also used \eqref{etajbound} and the assumption that  $\norm{}\xi_1 L f\norm{}_{r} < \infty$. Hence \eqref{apoint9boas} implies, for all $j$:
	\begin{align*}
		N(j) \leq E^{c+1}d^{-3c-2c^2}2^{c+3c^2+2c^3} \big(R(\xi_1 T f) + G(f)\big) (1+1/2+1/4+\ldots) = E^{c+1}d^{-3c-2c^2}2^{1+c+3c^2+2c^3} \big(R(\xi_1 T f) + G(f)\big).
	\end{align*}
	We may now put $j=j_0$ where $j_0 \geq 1$ is fixed and large, and use the definition of $N(j)$ to obtain
	\begin{equation}\label{hopefullylaststepiter}
		\norm{}\eta_{j_0} L f\norm{}_{r} \leq E^{c+1}d^{-3c-2c^2}2^{1+c+3c^2+2c^3+3j_0c^2+2j_0c^2} \big(R(\xi_1 T f) + G(f)\big).
	\end{equation}
	Since $\eta_{j_0} \equiv 1$ in a neighborhood of $supp(\xi)$, we note that $\norm{}\xi L f\norm{}_{r} =\norm{}\xi \eta_{j_0} L f\norm{}_{r}$. We estimate $\norm{}\xi \eta_{j_0} L f\norm{}_{r}$ from above by $\norm{}\xi\norm{}_{c}\norm{}\eta_{j_0} L f\norm{}_{r}$. Combining this with \eqref{hopefullylaststepiter} gives us the desired conclusion.
\end{proof}

\begin{remark}\label{remarkB2} {\rm We conclude this appendix with two remarks.
	\begin{enumerate}
		\normalfont
		\item One can actually prove a stronger version of this lemma where we start with $\norm{}\xi L f\norm{}_{r} \\ \leq E d^{-c} \norm{}\xi\norm{}_{c} \big(R(\xi_1 T f) + \norm{}\xi_1 L f\norm{}_{r - \epsilon}+ G(f) \big)$ instead of inequality \eqref{hypothesisestimate} and obtain the same conclusion. The proof of that is identical to the one above with a possible enlargement of the constant $E$ at various places.
		\item Although Lemma \ref{iterationschemeappend} is only formulated and proven for a domain $\Omega \subset \mathbb{C}^n$, the same conclusion also holds if $\Omega$ is a smoothly bounded, precompact domain $M$ in a complex manifold $M'$. There $U$ would be an open subset of $M'$ lying in some coordinate chart of $M'$ such that $U \cap M \neq \emptyset$. The proof is the same with Claim \ref{sobolevinterpolationthm} replaced by the general version in Remark \ref{Sobolevinterpolmanifold}.
	\end{enumerate}}
\end{remark}

\section{Proof of continuity of $N$ on the $H_s^{\dagger}(M)$ Sobolev spaces}
\label{appendixC}
 As in section \ref{genhuangli}, let $M$ be a precompact domain  with smooth pseudoconvex boundary of finite type (in the sense of D'Angelo) in a Hermitian manifold $(M', g')$ of dimension $n$. In the proof of Claim \ref{claim21} we need to use Lemma \ref{iterationschemeappend}in Appendix \ref{appendixB}. We need to take the operator $L$ in Appendix \ref{appendixB} to be the $\cdbar-$Neumann operator $N$ and, hence, we need to establish that $N$ satisfies \eqref{SobolevboundedL} (with $\Omega$ replaced by $M$). We prove this here in detail.

\begin{proposition}\label{continuitydbarneumodsobolev}
	For $s \geq 0$, $\norm{}Nu\norm{}_{s+2\lambda} \leq C_s \norm{}u\norm{}_{s}$ for $u \in H_s(M)$, where $\lambda$ is the subelliptic gain index in \eqref{pseudolocalestimateneg1norm}.  For $s > 0,$ we also have $\norm{}Nu\norm{}^{\dagger}_{-s} \leq C_s \norm{}u\norm{}^{\dagger}_{-s - 2 \lambda}$ for $u \in L^2_{p,q}(M) \subset H_{-s}^{\dagger}(M)$.
\end{proposition}

    \begin{proof}
		First, we prove the first part of the proposition for $s \geq 0$. Let us only prove it for $s =k \lambda$ where $k \in \mathbb{N} \cup \{0\}$, for simplicity. After that, the general $s \geq 0$ case follows from the interpolation theorem for operators on Sobolev spaces \cite[Theorem~B.3]{ChenandShawPDESCVbook}. We note that, by G\aa rding's inequality (\cite[Theorem~2.2.1]{FollandKohndbarneumanntheorybook}, \cite[Proposition~5.1.1]{ChenandShawPDESCVbook}), if $\eta \in C_c^{\infty}(M)$ is supported on an interior chart, then for any $f \in Dom(\cdbar) \cap Dom(\cdbarst)$, $$\norm{}\eta f\norm{}^2_{1} \leq E(\norm{}\cdbar f\norm{}^2_{0} + \norm{}\cdbarst f\norm{}^2_{0}+\norm{}f\norm{}^2_{0}),$$ and by \ref{eq:fullsobolevsubellipticseq0} of Appendix \ref{appendixA} it follows that, for $\eta$ supported in a boundary chart, we have
       	\begin{align*}
        	\norm{}\eta f\norm{}^2_{\lambda} \leq E (\norm{}\cdbar f\norm{}^2_{0}+ \norm{}\cdbarst f\norm{}^2_{0}+\norm{}f\norm{}^2_{0}).
        \end{align*}		
We shall complete the proof of the proposition by first proving two claims.
Suppose $f \in Dom(\cdbar) \cap Dom(\cdbarst)$. Since our manifold $M$ is precompact in $M'$, we can take a finite chart $\{U_i\}_{i=1}^N$ covering $\overline{M}$ and a partition of unity subordinate to this chart. Then according to the definition of Sobolev spaces in Section \ref{genhuangli} and the last two inequalities above,
		\begin{align}\label{globalsubelliptic}
			\norm{}f\norm{}^2_{\lambda, M} \leq E (\norm{}\cdbar f\norm{}^2_{0}+ \norm{}\cdbarst f\norm{}^2_{0}+\norm{}f\norm{}^2_{0}).
		\end{align}
		\begin{claim}\label{finiteharmonic}
			With $M$ as above of dimension $n$, for any $0 \leq p \leq n$ and $1 \leq q \leq n$, the space of harmonic $(p,q)$ forms $\mathcal{H}_{p,q}$ is finite dimensional.
		\end{claim}
		\begin{proof}
			Since $f \in \mathcal{H}_{p,q}$ if and only if $\cdbar f = \cdbarst f=0$, we have from \eqref{globalsubelliptic} that $\norm{}f\norm{}^2_{\lambda, M} \leq E \norm{}f\norm{}^2_{0}$ when $f \in \mathcal{H}_{p,q}$. Since, by Rellich Lemma \cite[Theorem~A.8 and Page~340]{ChenandShawPDESCVbook}, $H_{\lambda}(M)$ is compactly embedded in $L^2(M)$, we see from this last inequality that the $L^2$-unit ball in $\mathcal{H}_{p,q}$ is compact in $L^2(M)$. Thus $\mathcal{H}_{p,q}$ is finite dimensional.
		\end{proof}
Next, we state, without proof, a claim that is well-known in the field.
		\begin{claim}\label{laplaceesti}
			Let $\xi, \xi_1$ be real cut-off functions compactly supported near a point $p \in \overline{M}$ with $\xi_1|_{supp(\xi)} \equiv 1$. For $f \in Dom(\Box)$, where $\Box= \cdbar\cdbarst+\cdbarst\cdbar$, we have  for $s\geq 0$,
$$\norm{}\xi f\norm{}^2_{s+2\lambda} \leq E (\norm{}\xi_1 \Box f\norm{}^2_{s}+ \norm{}f\norm{}^2_{0}).
$$
		\end{claim}
		\begin{remark}
			\normalfont This follows from the classical work of Kohn-Nirenberg \cite{KohnNirenbergNonCoerciveboundaryvalue}. The proof is standard and readers are referred to (6.3), or rather the first inequality of page 472 of \cite{KohnNirenbergNonCoerciveboundaryvalue}. Note that in (6.3) ({\em loc.\ cit.}) the first two norms from the left are tangential Sobolev norms but on pages 475-476 of \cite{KohnNirenbergNonCoerciveboundaryvalue}, the authors indicate how to go from tangential Sobolev norm to total Sobolev norm, as we require here. Similar techniques of proof can also be found in the proof of Theorem 2.1.7 of \cite{FollandKohndbarneumanntheorybook}.
		\end{remark}

Now we are in a position to complete the proof of Proposition \ref{continuitydbarneumodsobolev}. Since $M$ is precompact in $M'$ take a finite covering $\{U_i\}_{i=1}^N$ of $\overline{M}$ and a partition of unity $\{\phi_i\}_{i=1}^N$ subordinate to $\{U_i\}_{i=1}^N$. Let us take cut-off functions $\{\psi_i\}_{i=1}^N$ such that each $\psi_i$ is compactly supported in $U_i$ and $\psi_i \equiv 1$ on a neighborhood of $supp(\phi_i)$.  Let $u \in H_s(M)$ and note that $Nu \in Dom(\Box)$. In the estimate of Claim \ref{laplaceesti}, putting $\xi=\phi_i, \xi_1= \psi_i$ and $f=Nu$, we obtain
        \begin{eqnarray*}
        	\begin{aligned}
        		\norm{}\phi_i Nu\norm{}^2_{s+2\lambda} &\leq E (\norm{}\psi_i \Box Nu\norm{}^2_{s}+ \norm{}Nu\norm{}^2_{0}) \\
        		& \leq E (\norm{}\psi_i \Box Nu\norm{}^2_{s}+ \norm{}u\norm{}^2_{0}) \\
        		&\leq E (\norm{}\psi_i (id-H)u\norm{}^2_{s}+ \norm{}u\norm{}^2_{0}) \\
        		&\leq E (\norm{}\psi_i u\norm{}^2_{s}+ \norm{}\psi_i H u\norm{}^2_{s}+ \norm{}u\norm{}^2_{0}) \\
        		&\leq E (\norm{}\psi_i u\norm{}^2_{s}+ \norm{}H u\norm{}^2_{s}+ \norm{}u\norm{}^2_{0}) \\
        		&\leq E (\norm{}\psi_i u\norm{}^2_{s}+ \norm{}H u\norm{}^2_{0}+ \norm{}u\norm{}^2_{0}) \\
        		&\leq E (\norm{}\psi_i u\norm{}^2_{s}+ \norm{}u\norm{}^2_{0}) \leq E \norm{}u\norm{}^2_{s}.
        	\end{aligned}
        \end{eqnarray*}
        The second inequality follows because the $\cdbar$-Neumann operator $N$ is a bounded operator from $L^2(M)$ to itself. The third inequality follows as $\Box N = id -H$, where $id$ is the identity operator and $H$ is the $L^2$ projection operator to the $L^2$ harmonic forms. The sixth inequality above follows as $Hg$ is a harmonic form and, as proved in Claim \ref{finiteharmonic}, the space of harmonic forms is finite dimensional and, hence, all norms on it are equivalent. The seventh inequality follows as $H$ is an $L^2$ projection operator. According to the definition of the Sobolev norms on $M$ in Section \ref{genhuangli}, the above proves the first part of Proposition \ref{continuitydbarneumodsobolev}. For proving the second part, we recall that, for $s>0$, we have defined $\norm{}u\norm{}_{-s}^{\dagger}=\sup_{g \in C^{\infty}_{p,q}(\overline{M}) \cap Dom(\cdbarst)} |(u,g)|/\norm{}g\norm{}_{s}$. For any $g \in C^{\infty}_{p,q}(\overline{M}) \cap Dom(\cdbarst),$ we note that, for $s >0$,
		\begin{align*}
			|(Nu,g)| =|(u,Ng)| &\leq \norm{}u\norm{}^{\dagger}_{-s-2\lambda} \norm{}Ng\norm{}_{s+ 2 \lambda} \\
			&\leq E \norm{}u\norm{}^{\dagger}_{-s-2\lambda} \norm{}g\norm{}_{s}.
		\end{align*}
		Here the first equality follows as the $\cdbar$-Neumann operator $N$ is self-adjoint. The inequality after follows from the definition of $\norm{} \cdot \norm{}^{\dagger}_{-s-2\lambda}$ and the fact that if $g \in C^{\infty}(\overline{M})$ then $Ng \in C^{\infty}(\overline{M}) \cap Dom(\cdbarst)$. The last inequality follows from the first assertion of Proposition \ref{continuitydbarneumodsobolev}, which we just proved. The constant $E$ is independent of $g$. Hence we have $\norm{}Nu\norm{}^{\dagger}_{-s} \leq E \norm{}u\norm{}^{\dagger}_{-s-2\lambda}$, as claimed.
	\end{proof}
	
\bibliography{weakpscvxbiblio}
\bibliographystyle{amsalpha.bst}
\end{document}